\documentclass[11pt,reqno,oneside]{amsart}
\usepackage{amsfonts,amssymb,latexsym,xspace,epsfig,graphics,color}
\usepackage{amsmath,enumerate,stmaryrd,xy}
\usepackage{bbm}
\usepackage{subfigure}
\usepackage{hyperref}
\usepackage[authoryear]{natbib}
\usepackage{hypernat}
\usepackage{appendix}
\usepackage{xspace}

\newcommand{\E}{\mathbb{E}}

\newcommand{\N}{\ensuremath{\mathbb{N}}}

\newcommand{\Z}{\ensuremath{\mathbb{Z}}}    
\renewcommand{\P}{\ensuremath{\mathbb{P}}}   

\newcommand{\egaldef}{:=} 
\newcommand{\telque}{\, \mbox{ s.t. } \,} 

\newcommand{\paren}[1]{\left( \left. #1 \right. \right)} 
\newcommand{\croch}[1]{\left[ \left. #1 \right. \right]} 
\newcommand{\set}[1]{\left\{ \left. #1 \right. \right\}}
\newcommand{\absj}[1]{\left\lvert #1 \right\rvert} 

\renewcommand{\P}{\mathbb{P}}

\newcommand{\ulx}{\ensuremath{\underline{x}}}

\newcommand{\Qr}[1]{\Pr\left(#1\right)}

\newcommand{\Fcal}{\ensuremath{\mathcal{F}}}

\newcommand{\bN}{\ensuremath{\overline{\N}}}

\usepackage{algorithmic}
\usepackage{algorithm}

\newtheorem{remark}{Remark}[section]

\newtheorem{lemma}{Lemma}[section]

\newtheorem{obs}{Observation}[section]
\newtheorem{fact}{Fact}[section]
\newtheorem{defi}{Definition}[section]
\newtheorem{theo}{Theorem}[section]
\newtheorem{prop}{Proposition}[section]
\newtheorem{coro}{Corollary}[section]

\newtheorem{notation}{Notation}[section]

\begin{document}
\DeclareGraphicsExtensions{.pdf,.gif,.jpg}

\keywords{Chains of infinite order, coupling from the past algorithms, canonical Markov approximation, $\bar{d}$-distance}
\subjclass[2000]{Primary 60G10; Secondary 60K99}

\title{Markov approximation of chains of infinite order in the $\bar{d}$-metric}




\author{S. Gallo}
\address{Instituto de Matem\'atica\\Estat\'{\i}stica e Computa\c c\~ao
Cient\'\i fica\\ 
Universidade Estadual de Campinas\\
Rua Sergio Buarque de Holanda, 651\\
13083-859 Campinas, Brasil}
\email{gsandro@ime.unicamp.br}


\author{M. Lerasle}
\address{Laboratoire J.A.Dieudonn\'e  UMR CNRS 6621\\
Universit\'e de Nice Sophia-Antipolis, Parc Valrose\\
06108 Nice Cedex 2\\
France }
\email{ mlerasle@unice.fr}

\author{D. Y. Takahashi}
\address{Institute of Neuroscience and Psychology Department, Princeton University\\
 Green Hall, NJ, 08540}
\email{takahashiyd@gmail.com}
\thanks{SG is supported by FAPESP grant 2009/09809-1. ML was supported by FAPESP grant 2009/09494-0. DYT was supported by FAPESP grant 2008/08171-0 and Pew Latin American Fellowship. This work is part of USP project ``Mathematics, computation, language and the brain.''}

\begin{abstract}
We derive explicit upper bounds for the $\bar{d}$-distance between a chain of infinite order and its canonical $k$-steps Markov approximation. Our proof is entirely constructive and involves a ``coupling from the past'' argument. The new method covers non necessarily continuous probability kernels, and chains with null transition probabilities.  These results imply in particular the Bernoulli property for these processes.

\end{abstract}

\maketitle

\section{{\bf INTRODUCTION}}

Chains of infinite order are random processes that are specified by probability kernels (conditional probabilities), which may depend on the whole past. They provide a flexible model that is very useful in different areas of applied probability and statistics, from  bioinformatics \cite{bejerano2001a,busch2009} to linguistics \cite{galves2009,GGGL}. They are also models of considerable theoretical interest in ergodic theory \cite{walters/2007, coelho/quas/1998, quas/1996, hulse/1991} and in the general theory of stochastic process \cite{fernandez/maillard/2005, comets/fernandez/ferrari/2002, bramson/kalikow/1993}. 
A natural approach to study chains of infinite order is to approximate the original process by Markov chains of growing orders. 
In this article, we derive new upper-bounds on the $\bar{d}$-distance between a chain and its canonical $k$-steps Markov approximation.

\vspace{0.2cm}

Introduced by \cite{ornstein/1974} to study the isomorphism problem for Bernoulli shifts, the $\bar{d}$-metric is of fundamental importance in ergodic theory where chains of infinite order are also known as $g$-measures. The $\bar{d}$-distance between two processes can be informally described as the minimal proportion of times we have to change a typical realization of one process in order to obtain a typical realization of the other. \cite{ornstein/1974} showed that the set of processes which are measure theoretic isomorphic to Bernoulli shifts is $\bar{d}$-closed. Ergodic Markov chains are examples of processes that are isomorphic to Bernoulli shifts. Therefore,  if a process can be approximated arbitrary well under the $\bar{d}$-metric by a sequence of ergodic Markov chains, then this process has the Bernoulli property. In this article we prove the existence of Markov approximation schemes for classes of chains of infinite order with non-necessary continuous and with possibly null transition probabilities. Some of these processes were not considered before.
For example, \cite{coelho/quas/1998}, \cite{fernandez/galves/2002}, and \cite{johansson/oberg/pollicott/2010} required the continuity of the probability kernels. Our results show that these new examples are isomorphic to Bernoulli shifts and provide explicit upper bounds for the Markov approximation in several important cases, giving therefore information on \emph{how good} these approximations are.  

\vspace{0.2cm}

Besides ergodic theory, the $\bar d$-distance is useful in statistics and information theory. \cite{rissanen/1983} proposed to model data as realizations of stochastic chains, and proved that these data can be optimally compressed using the (unknown) probability kernel of the chain. The statistical problem is then to recover this probability kernel from the observation of typical data. Since the number of parameters to estimate is infinite, this task is impossible in general. A possible strategy to overcome this problem is the following.
(1) Couple the original chain with a Markov approximation
and (2) work with the approximating Markov chain. The $\bar{d}$-distance between the chain and its Markov approximation controls the error made in step (1). The idea is that, if this control is good enough, the good properties of the approximating Markov chain proved in step (2) can be used to study the original chain.  For instance, \cite{duarte/galves/garcia/2006} and \cite{csiszar/talata/2010} derived consistency results for chains of infinite order from the consistency of BIC estimators for Markov chains proved in \cite{csiszar/talata/2006}. This ``two steps'' procedure was also used in \cite{collet/duarte/galves/2005} to obtain a bootstrap central limit theorem for chains of infinite order from the renewal property of the approximating Markov chains.

\vspace{0.2cm}

Our main results derive from coupling arguments. We first introduce a flexible class of \emph{Coupling from the past}  algorithms (CFTP algorithms, see Section \ref{sec:cftp}). CFTP algorithms constitute an important class of perfect simulation algorithms popularized by \cite{propp/wilson/1996}. Our main assumption on the chain is that the original chain of infinite order can be perfectly simulated \emph{via} such CFTP algorithms. We state a technical result, Lemma \ref{theo:1}, which provides an abstract upper bound for the $\bar{d}$-distance with the canonical Markov approximation. This bound is then made explicit under various extra assumptions on the process used in the study of the CFTP algorithms of \citep{comets/fernandez/ferrari/2002,gallo/2009,desantis/piccioni/2010,gallo/garcia/2011}. 

%
%
%
\vspace{0.2cm}

To our knowledge, \cite{fernandez/galves/2002} provide the best explicit bounds in the literature for the $\bar{d}$-distance between a chain of infinite order and its canonical Markov approximation, depending only on the continuity rate of the probability kernels. Their result applies to weakly non-null chains having summable continuity rates. Our method recovers the same bounds, substituting  weak non-nullness by a weaker assumption, see Theorem \ref{coro:expli2}. Assuming weak non-nullness, we also obtain explicit upper bounds in some non-summable continuity regimes and other not even necessarily continuous, but satisfying certain types of localized continuity, as introduced in \cite{desantis/piccioni/2010}, \cite{gallo/2009} and \cite{gallo/garcia/2011}. This is the content of Theorems \ref{coro:italia} and \ref{lemma:1} which provide, as far as we know, the first results for non-continuous chains. Our results should also be compared with the results in \cite{johansson/oberg/pollicott/2010}, where they prove the Bernoulli property for square summable continuity regime assuming strong non-nullness, although they don't provide an explicit upper bound for the approximations. 

\vspace{0.2cm}

The paper is organized as follows. In Section \ref{sec:notations}, we introduce the notation and basic definitions used all along the paper. In Section \ref{sec:conccoup}, we construct the coupling between the original chain and its canonical Markov approximation and we introduce the class of CFTP algorithms perfectly simulating the chains. Our main results are stated in Section \ref{sec:theo}.
We postpone the proofs to Section \ref{sec:proofs}. For convenience of the reader, we leave in Appendix some extensions and technical results on the ``house of cards'' processes that are useful in our applications and are of independent interest.

\section{{\bf NOTATION, DEFINITIONS AND BACKGROUND}}\label{sec:notations}

\subsection{Notation}We use the conventions that $\N^{*}=\N\setminus\set{0}$, $\overline{\N}=\N^{*}\cup\set{\infty}$. Let $A$ be the set $\set{1,2,\ldots,N}$ for some $N\in\overline{\N}$.  Given two integers $m\leq n$, let $a_m^n$ be the string $a_m \ldots a_n$ of symbols in $A$. For
any $m\leq n$, the length of the string $a_m^n$ is denoted by
$|a_m^n|$ and defined by $n-m+1$. Let $\emptyset$ denote the empty string, of length $|\emptyset|=0$. For any $n\in\mathbb{Z}$, we will
use the convention that $a_{n+1}^{n}=\emptyset$, and naturally
$|a_{n+1}^{n}|=0$. Given two strings $v$ and $v'$, we denote by $vv'$
the string of length $|v| + |v'| $ obtained by concatenating the two
strings. If $v'=\emptyset$, then $v\emptyset=\emptyset v=v$. The concatenation of strings is also extended to the case
where $v=\ldots a_{-2}a_{-1}$ is a semi-infinite sequence of
symbols. If $n\in\N^{*}$ and $v$ is a finite string of
symbols in $A$,  $v^{n}=v\ldots v$ is the concatenation of
$n$ times the string $v$. In the case where $n=0$, $v^{0}$ is the empty string $\emptyset$. Let
$$
A^{-\mathbb{N}}=A^{\{\ldots,-2,-1\}}\,\,\,\,\,\,\textrm{ and }\,\,\,\,\,\,\,  A^{\star} \,=\, \bigcup_{j=0}^{+\infty}\,A^{\{-j,\dots, -1\}}\, ,
$$
be, respectively, the set of all infinite strings of past symbols and the set of all finite strings of past symbols. 
The case $j=0$ corresponds to the empty string $\emptyset$. Finally, we denote by    $\underline{a}=\ldots a_{-2}a_{-1}$ the elements of $A^{-\mathbb{N}}$.

\subsection{Kernels, chains and coupling}

\begin{defi}A family of transition probabilities, or \emph{kernel}, on an alphabet $A$ is a function 
\begin{equation*}
\begin{array}{cccc}
P:&A\times A^{-\mathbb{N}}&\rightarrow& [0,1]\\
&(a,\underline{x})&\mapsto&P(a|\underline{x})
\end{array}
\end{equation*}
such that
\[
\sum_{a\in A}P(a|\underline{x})=1\,\,,\,\,\,\,\,\,\forall \underline{x}\in A^{-\mathbb{N}}.
\] 
\end{defi}
\noindent $P$ is called a Markov kernel if there exists $k$ such that $P(a|\ulx)=P(a|\underline{y})$ when $x_{-k}^{-1}=y_{-k}^{-1}$. In the present paper we are mostly interested in \emph{non}-Markov kernels, in which $P(a|\underline{x})$ may depend on the whole past $\underline{x}$.
\begin{defi}
A stationary stochastic chain ${\bf X}=\{X_{n}\}_{n\in\Z}$ with distribution $\mu$ on $A^{\mathbb{Z}}$ is said to be \emph{compatible} with a family of transition probabilities $P$ if the later is a regular version of the conditional probabilities of the former, that is
\begin{equation}\label{compa}
\mu(X_{0}=a|X_{-\infty}^{-1}=\underline{x})=P(a|\underline{x})
\end{equation}
for every $a\in A$ and $\mu$-a.e. $ \underline{x}$ in $A^{-\mathbb{N}}$. 
\end{defi}
\noindent If $P$ is non-Markov, it may be hard to prove the existence of a stationary chain ${\bf X}$ compatible with it. In order to solve this issue, we assume the existence of \emph{coupling from the past algorithms} for the chain (see Section \ref{sec:cftp}). This ``constructive argument'' garantees the existence and uniqueness of ${\bf X}$. 

\begin{defi}[Canonical $k$-steps Markov approximation]Assume that ${\bf X}$ is a stationary chain with distribution $\mu$. The \emph{canonical $k$-steps Markov approximation} of ${\bf X}$ is the stationary $k$-step Markov chain ${\bf X}^{[k]}$ compatible with the kernel $P^{[k]}_{\mu}$ defined as
\[
P^{[k]}_{\mu}(a|x_{-k}^{-1})=\mu(X_{0}=a|X_{-k}^{-1}=x_{-k}^{-1}).
\]
\end{defi}
\noindent Since $\mu$ is uniquely determined by $P$, we will not mention any more the subscript $\mu$ in $P^{[k]}_{\mu}$, it will be understood that $P^{[k]}=P^{[k]}_{\mu}$.

\vspace{0.2cm}

Let us recall that a coupling between two chains ${\bf X}$ and ${\bf Y}$ taking  values in the same alphabet $A$ is a stochastic chain ${\bf Z}=\{Z_{n}\}_{n\in\mathbb{Z}}=\{(\bar{X}_{n},\bar{Y}_{n})\}_{n\in\mathbb{Z}}$ on $(A\times A)^{\Z}$ such that $\bar{\bf X}$ has the same distribution as ${\bf X}$ and $\bar{\bf Y}$ has the same distribution as ${\bf Y}$. For any pair of stationary chains ${\bf X}$ and ${\bf Y}$, let $\mathcal{C}({\bf X},{\bf Y})$ be the set of couplings between ${\bf X}$ and ${\bf Y}$.
\begin{defi}[$\bar{d}$-distance]The $\bar{d}$-distance between two stationary chains ${\bf X}$ and ${\bf Y}$ is defined  by
\[
\bar{d}({\bf X},{\bf Y})=\inf_{(\bar{\bf X},\bar{\bf Y})\in\mathcal{C}({\bf X},{\bf Y})}\mathbb{P}(\bar{X}_{0}\neq\bar{Y}_{0}).
\]
\end{defi}

For the class of ergodic processes, this distance has another interpretation which is more intuitive: it is the minimal proportion of sites we have to change in a \emph{typical} realization of ${\bf X}$ in order to obtain a \emph{typical} realization of ${\bf Y}$. Formally, 
\[
\bar{d}({\bf X},{\bf Y})=\inf_{(\bar{{\bf X}},\bar{\bf Y})\in\mathcal{C}({\bf X},{\bf Y})}\lim_{n\rightarrow+\infty}\frac{1}{n}\sum_{i=1}^{n}{\bf 1}\{\bar{X}_{i}\neq \bar{Y}_{i}\}.
\] 

\subsection{Coupling from the past algorithm (CFTP)}\label{sec:cftp}

Our CFTP algorithm constructs a sample of the stationary chain compatible with a given kernel $P$, using a sequence ${\bf U} = \{U_n\}_{n \in \Z}$ of i.i.d. random variables uniformly distributed in $[0,1[$. We denote by $(\Omega,\mathcal{F},\mathbb{P})$ the probability space associated to ${\bf U}$. The CFTP is completely determined by its \emph{update function} $F:A^{-\mathbb{N}}\cup A^{\star}\times[0,1[\rightarrow A$ which satisfies that, for any $\underline{a}\in A^{-\mathbb{N}}$ and for any $a\in A$, $\mathbb{P}(F(\underline{a},U_{0})=a)=P(a|\underline{a})$. Using this function, we define the \emph{set of coalescence times} $\Theta$ and the \emph{reconstruction function} $\Phi$ associated to $F$. For any pair of integers $m,n$ such that $-\infty<m\le n<+\infty$, let $F_{\{m,n\}}(\underline{a},U_{m}^{n})\in A^{n-m+1}$ be the sample obtained by \emph{applying recursively $F$ on the fixed past $\underline{a}$}, i.e, let $F_{\{m,m\}}(\underline{a},U_{m}):=F(\underline{a},U_{m})$ and 
\[
F_{\{m,n\}}(\underline{a},U_{m}^{n}):=F_{\{m,n-1\}}(\underline{a},U_{m}^{n-1})F(\underline{a}F_{\{m,n-1\}}(\underline{a},U_{m}^{n-1}),U_{n})\;.
\]
Secondly, let $F_{[m,m]}(\underline{a},U_{m}):=F(\underline{a},U_{m})$ and 
\begin{equation}\label{eq:F}
F_{[m,n]}(\underline{a},U_{m}^{n})=F\left(\underline{a}F_{\{m,n-1\}}(\underline{a},U_{m}^{n-1}),U_{n}\right)\;.
\end{equation}
$F_{[m,n]}(\underline{a},U_{m}^{n})$ is the last symbol of the sample $F_{\{m,n\}}(\underline{a},U_{m}^{n})$. The set
\begin{equation}\label{eq:theta}
\Theta[n]:=\{j\leq n:F_{[j,n]}(\underline{a},U_{j}^{n}) = F_{[j,n]}(\underline{b},U_{j}^{n}) \,\,\,\textrm{for all }\,\underline{a},\underline{b} \in A^{-\N} \}
\end{equation}
is called the \emph{set of coalescence times} for the time index $n$. 
Finally, the reconstruction function of time $n$ is defined by
\begin{equation}\label{eq:Phi}
[\Phi({\bf U})]_{n}=F_{[\theta[n],n]}(\underline{a},U_{\theta[n]}^{n})
\end{equation}
where $\theta[n]$ is any element of $\Theta[n]$. 
Given a kernel $P$, if $\Theta[0]\neq \emptyset$ a.s. and therefore $\Theta[n]\neq\emptyset$ a.s. for any $n\in\mathbb{Z}$, then  $[\Phi({\bf U})]_{n}$ is distributed according to the unique stationary measure compatible with $P$, see \cite{desantis/piccioni/2010}. 
%
%

\section{{\bf CONSTRUCTION OF THE COUPLING}}\label{sec:conccoup}

For any $\underline{a}\in A^{-\N}$, let $\mathcal{I}(\underline{a}):=\{I(a|a_{-k}^{-1})\}_{k\in\bN,\;a\in A}$ be any partition of $[0,1[$  having the following properties:

\begin{enumerate}
\item For any $k\in\bN$, the Lebesgue measure or length $|I_{k}(a|a_{-k}^{-1})|$ of $I_{k}(a|a_{-k}^{-1})$ only depends on $a$ and $a_{-k}^{-1}$,
\item for any $\underline{a}$ and $a$
\[
\sum_{k\in\bN}|I_{k}(a|a_{-k}^{-1})|=P(a|\underline{a}),
\]
\item the intervals are disposed  as represented in the upper part of Figure \ref{fig:partition}.
\end{enumerate}
\begin{figure}[htp]
\centering
\includegraphics{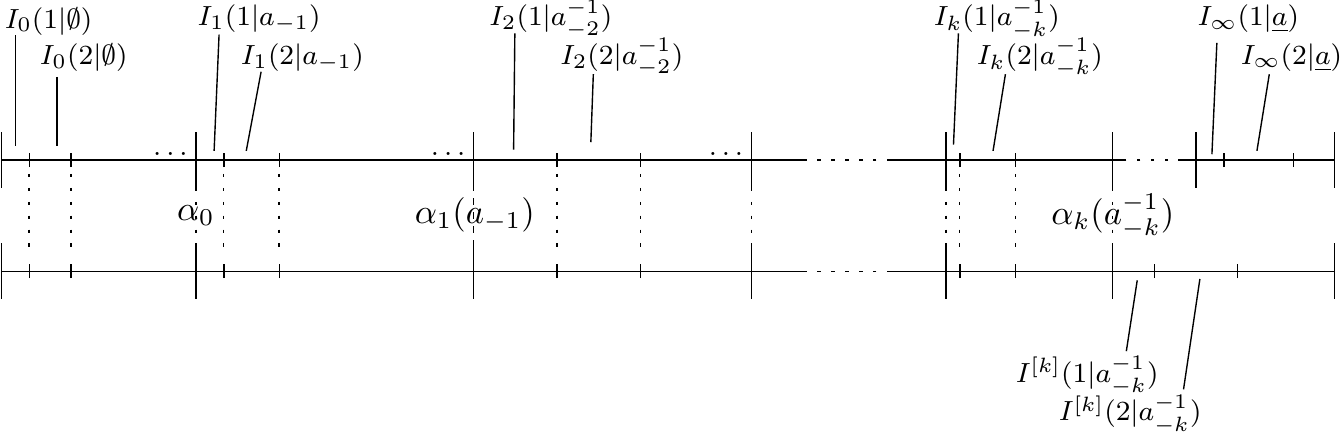}
\caption{Illustration of a range partition related to some infinite past $\underline{a}$. The upper partition is the one used for the original kernel $P$, whereas the one below is used for the approximating kernel $P^{[k]}$.}
\label{fig:partition}
\end{figure}

\begin{defi}
We call \emph{range partitions} the partitions of $[0,1[$ satisfying (1), (2) and (3) for some kernel $P$.
\end{defi}

\noindent 
The following lemma is proved in Section~\ref{sect.proof.lemma:simple}. 

\begin{lemma}\label{lemma:simple}
A set of range partitions satisfies, for any $\underline{a}$ and $a\in A$,
\[
\sum_{i=0}^{k}|I_{i}(a|a_{-i}^{-1})|\leq \inf_{\underline{z}}P(a|a_{-k}^{-1}\underline{z})\,,\,\,\,\,\forall k\geq0\enspace.
\]
\end{lemma}

\noindent
Given a range partition $\mathcal{I}(\underline{a})$, the following $F$ is an update function, due to property (2).
\begin{equation}\label{eq:update}
F(\underline{a},U_{0}):=\sum_{a\in A}a.{\bf 1}\left\{U_{0}\in \bigcup_{k\in\bN}I_{k}(a|a_{-k}^{-1})\right\}.
\end{equation}

\noindent
This function $F$ explains the name ``range partition'': for a given past $\underline{a}$, when the uniform r.v. $U_{0}$ belongs to $\bigcup_{a\in A}\bigcup_{i=0}^{k}I_{i}(a|a_{-i}^{-1})$, then $F$ constructs a symbol looking at a range $\leq k$ in the past. 

\noindent
Let $L:A^{-\mathbb{N}}\cup A^{\star}\times[0,1[\rightarrow \mathbb\{0,1,2,\ldots\}$ be the \emph{range function} defined by
\begin{equation}\label{eq:length}
L(\underline{a},u):=\sum_{k\in\bN}k.{\bf1}\{u\in \cup_{a\in A}I_{k}(a|a_{-k}^{-1})\}\;.
\end{equation}
$L$ associates to a past $\underline{a}$ and a real number $u\in [0,1[$ the length of the suffix of $\underline{a}$ that $F$ needs in order to construct the next symbol when $U_{0}=u$. 

\noindent 
Using these functions, define, as in Section \ref{sec:cftp}, the related coalescence sets $\Theta[i]$, $i\in\mathbb{Z}$, and the reconstruction function $\Phi({\bf U})$, which is distributed according to the unique stationary distribution compatible with $P$ whenever $\Theta[0]$ is a.s. non-empty. \\

Let us now define the functions $F^{[k]}$ and $L^{[k]}$ that we will use for the construction of ${\bf X}^{[k]}$. 
Observe that, on the one hand, by definition of the canonical $k$-steps Markov approximation we have for any $a\in A$ and $a_{-k}^{-1}\in A^{k}$
\[
P^{[k]}(a|a_{-k}^{-1}):=\mu(X_{0}=a|X_{-k}^{-1}=a_{-k}^{-1})=\int_{A^{-\mathbb{N}}}P(a|a_{-k}^{-1}\underline{z})d\mu(\underline{z}|a_{-k}^{-1})\geq \inf_{\underline{z}}P(a|a_{-k}^{-1}\underline{z})\enspace.
\]
On the other hand, by Lemma \ref{lemma:simple}, $\inf_{\underline{z}}P(a|a_{-k}^{-1}\underline{z})\geq\sum_{j=0}^{k}|I_{k}(a|a_{-k}^{-1})|$. Thus we can define, for any $a_{-k}^{-1}$, the set of intervals $\{I^{[k]}(a|a_{-k}^{-1})\}_{a\in A}$ having length $|I^{[k]}(a|a_{-k}^{-1})|=P^{[k]}(a|a_{-k}^{-1})-\sum_{j=0}^{k}|I_{k}(a|a_{-k}^{-1})|$ and disposed as in Figure \ref{fig:partition}. The functions $F^{[k]}$ and $L^{[k]}$ are defined as follows
\begin{equation}\label{eq:up}
F^{[k]}(a_{-k}^{-1},U_{0}):=\sum_{a\in A}a{\bf 1}\{U_{0}\in \cup_{j=0}^{k}I_{j}(a|a_{-j}^{-1})\cup I^{[k]}(a|a_{-1}^{-k})\}
\end{equation}
and
\begin{equation}
L^{[k]}(a_{-k}^{-1},U_{0}):=
\sum_{j=0}^{k}j.{\bf1}\{U_{0}\in\cup_{a\in A}I_{j}(a|a_{-j}^{-1})\}+k.{\bf1}\{U_{0}\in\cup_{a\in A}I^{[k]}(a|a_{-1}^{-k})\}.
\end{equation}
Using these functions, define, as in Section \ref{sec:cftp}, the related coalescence sets $\Theta^{[k]}[i]$, $i\in\mathbb{Z}$, and the reconstruction function $\Phi^{[k]}({\bf U})$, which is distributed according to the unique stationary distribution compatible with $P^{[k]}$ whenever $\Theta^{[k]}[0]$ is a.s. non-empty. \\

Using the same sequence of uniforms ${\bf U}$ and assuming that $\Theta[0]$ and $\Theta^{[k]}[0]$ are a.s. non-empty, $(\Phi({\bf U}),\Phi^{[k]}({\bf U}))$ is a $(A\times A)$-valued chain with coordinates distributed as ${\bf X}$ and ${\bf X}^{[k]}$ respectively.  It follows that $(\Phi({\bf U}),\Phi^{[k]}({\bf U}))$ is a coupling between both chains. Hence, we have constructed a CFTP algorithm for perfect simulation of the coupled chains.

\section{{\bf STATEMENTS OF THE RESULTS}}\label{sec:theo}
%

\subsection{A key lemma}

Let us first state a technical lemma that is central in the proof of our main results.

\begin{lemma}\label{theo:1}
Assume that there exists a set of range partitions $\{\mathcal{I}(\underline{a})\}_{\underline{a}}$ such that the sets of coalescence times $\Theta[0]\cap\Theta^{[k]}[0]$ is $\mathbb{P}$-a.s. non-empty. Then, for any $\theta[0]\in \Theta[0]$,
\begin{equation}\label{eq:discrepancy}
\bar{d}({\bf X},{\bf X}^{[k]})\leq \mathbb{P}\left(\bigcup_{\underline{a}}\bigcup_{i=\theta[0]}^{0}\left\{L\left(\underline{a}F_{\{\theta[0],i-1\}}(\underline{a},U_{\theta[0]}^{i-1}),U_{i}\right)>k\right\}\right).
\end{equation}
where, for $i=\theta[0]$, the event reads $\{L\left(\underline{a},U_{\theta[0]}\right)>k\}$.
\end{lemma}
%
%

Examples of range partitions satisfying the conditions of this theorem have already been built, for example in \cite{comets/fernandez/ferrari/2002}, \cite{gallo/2009}, \cite{gallo/garcia/2011} and \cite{desantis/piccioni/2010}. These works assume some regularity conditions on $P$ and some non-nullness hypothesis which are presented in Sections \ref{sec:continuity}, \ref{sec:ita}, \ref{sec:localconti}. In these sections, we derive explicit upper bounds for \eqref{eq:discrepancy} under the respective assumptions. Before that, let us give an interesting remark on Bernoullicity.

\begin{obs}[A remark on Bernoullicity]
In the conditions of each works cited above, we will exhibit $\theta[0]\in\Theta[0]$ which belongs to $\Theta^{[k]}[0]$ for any sufficiently large $k$'s, and we will prove that 
\begin{equation}\label{eq:suppose}
\mathbb{P}\left(\bigcup_{\underline{a}}\bigcup_{i=\theta[0]}^{0}\left\{L\left(\underline{a}F_{\{\theta[0],i-1\}}(\underline{a},U_{\theta[0]}^{i-1}),U_{i}\right)>k\right\}\right)\stackrel{k\rightarrow\infty}{\longrightarrow}0.
\end{equation}
It follows,  by Lemma \ref{theo:1}, that
\begin{equation*}
\lim_{k \rightarrow \infty} \bar{d}({\bf X},{\bf X}^{[k]}) = 0.
\end{equation*}
We also have, for any sufficiently large $k$'s, that ${\bf X}^{[k]}$ is an ergodic Markov chain since $\Theta^{[k]}[0]$ is non-empty. Now, by the $\bar{d}$-closure of the set of processes isomorphic to a Bernoulli shift (see for example \cite{shields/1996} Theorem IV.2.10, p.228) and the fact that ergodic Markov processes have the Bernoulli property (\cite{shields/1996} Theorem IV.2.10, p.227), we conclude that the processes considered in \cite{comets/fernandez/ferrari/2002}, \cite{gallo/2009}, \cite{gallo/garcia/2011} and \cite{desantis/piccioni/2010} have the Bernoulli property.
\end{obs}
%
%

\subsection{Kernels with summable continuity rate}\label{sec:continuity}
Let us first define continuity.
\begin{defi}[Continuity points and continuous kernels]\label{defi:conti}For any $k\in\N$, $a$ and $a_{-k}^{-1}$, let $\alpha_{k}(a|a_{-k}^{-1}):=\inf_{\underline{z}}P(a|a_{-k}^{-1}\underline{z})$. A past $\underline{a}$ is called a \emph{continuity point} for $P$ or $P$ is said to be continuous in $\underline{a}$ if
\[
\alpha_{k}(a_{-k}^{-1}):=\sum_{a\in A}\alpha_{k}(a|a_{-k}^{-1})\stackrel{k\rightarrow+\infty}{\longrightarrow}1.
\] 
We say that $P$ is continuous when 
\[
\alpha_{k}:=\inf_{a_{-k}^{-1}}\alpha_{k}(a_{-k}^{-1})\stackrel{k\rightarrow+\infty}{\longrightarrow}1.
\]
We say  that $P$ has summable continuity rate when $\sum_{k\geq 0}(1-\alpha_{k})<\infty$.
\end{defi}

\noindent
%
\noindent
Let us also define weak non-nullness. 
\begin{defi}\label{def:weakly}
We say that a kernel $P$ is \emph{weakly non-null} if $\alpha_{0}>0$, where $\alpha_{0}:=\sum_{a\in A}\alpha_{0}(a|\emptyset)$.
\end{defi}

\noindent
\cite{desantis/piccioni/2010} have introduced a more general assumption that we call \emph{very weak non-nullness}, see Definition \ref{def.vwnn}. We postpone this definition to Section \ref{sect:proof.thm.expli2} in order to avoid technicality at this stage.



\begin{theo}\label{coro:expli2}
Assume that $P$ has summable continuity rate and is very weakly non-null. Then, there exists a constant $C<+\infty$ such that, for any sufficiently large $k$,
\[
\bar{d}({\bf X},{\bf X}^{[k]})\leq C(1-\alpha_{k})\enspace.
\]
\end{theo}
\begin{remark}
This upper bound is new since we do not assume weak non-nullness. \cite{fernandez/galves/2002} showed, under weak non-nullness, that for any sufficiently large $k$, there exists a positive constant $C$ such that
\[
\bar{d}({\bf X},{\bf X}^{[k]})\leq C\beta_{k}
\]
where 
\[
\beta_{k}:=\sup\{|P(a|a_{-k}^{-1}\underline{x})-P(a|a_{-k}^{-1}\underline{y})|\,:\,\,a\in A,\,a_{-k}^{-1}\in A^{k},\,\underline{x},\underline{y}\in A^{-\mathbb{N}}\}.
\]
This later quantity is related to $\alpha_{k}$ through the inequalities $1-\alpha_{k}\ge \frac{|A|}{2}\beta_{k}$ and $1-\alpha_{k}\leq D\beta_{k}$ for some $D>1$ and sufficiently larges $k$'s. Moreover, $1-\alpha_{k}=\beta_{k}$ for binary alphabets. Thus, Theorem \ref{coro:expli2} extends the bound in \cite{fernandez/galves/2002}. 
\end{remark}

\subsection{Using a prior knowledge of the histories that occur}\label{sec:ita}

\cite{desantis/piccioni/2010} introduced the following assumptions on kernels. Define 
\[
\forall k\geq 1,\quad J_{k}(U_{-k}^{-1}):=\{\underline{x}\in A^{-\N}\telque \forall 1\leq l\leq k,\;x_{-l}=a\quad\textrm{if}\,\,U_{-l}\in I(a|\emptyset)\,\mbox{for some}\,a\in A\},
\]
\[
A_{0}:=\alpha_{0}\quad\mbox{and}\quad \forall k\geq1, \quad A_{k}(U_{-k}^{-1}):=\inf\{\alpha_{k}(x_{-k}^{-1}):\underline{x}\in J_{k}(U_{-k}^{-1})\} \enspace.
\]
Finally, let
\begin{equation}\label{eq:ell}
\ell(U_{-\infty}^{0}):=\inf\{j\in\mathbb{N}:U_{0}<A_{j}(U_{-j}^{-1})\}.
\end{equation}

\begin{theo}\label{coro:italia}
If ${\bf X}$ has a kernel that satisfies $\E\paren{\prod_{k\ge0}A_{k}(U_{-k}^{-1})^{-1}}<\infty$, then there exists a positive constant $C<+\infty$ such that
\[
\bar{d}({\bf X},{\bf X}^{[k]})\leq C\mathbb{P}(\ell(U_{-\infty}^{0})>k).
\] 
\end{theo}

In order to illustrate the interest of this result, let us give two simple examples. Other examples can be found in \cite{desantis/piccioni/2010} and \cite{gallo/garcia/2011}.

\vspace{0.3cm}

\paragraph{{\bf Summable continuity regime with weak non-nullness}} Theorem \ref{coro:italia} allows to recover the result of Theorem \ref{coro:expli2} in the weakly non-null case. To see this, it is enough to observe that, for any $U_{-k}^{-1}$, $A_{k}(U_{-k}^{-1})\geq \alpha_{k}$ (see Definition \ref{defi:conti} for $\alpha_{k}$). It follows that $\prod_{k\ge0}\alpha_{k}>0$ (which is equivalent to $\sum_{k\ge0}(1-\alpha_{k})<+\infty$), implies that $\prod_{k\ge0}A_{k}(U_{-k}^{-1})$ is bounded away from zero, hence, its inverse has finite expectation. Hence, Theorem \ref{coro:italia} applies and gives
\[
\bar{d}({\bf X},{\bf X}^{[k]})\leq C\mathbb{P}(\ell(U_{-\infty}^{0})>k)\leq C(1-\alpha_{k})\enspace.
\]

\vspace{0.3cm}

\paragraph{{\bf A simple discontinuous kernel on $A=\{1,2\}$}} Let $\epsilon\in(0,1/2)$ and let $\{p_{i}\}_{i\ge0}$ be any sequence such that,  $\epsilon\leq p_{i}<1-\epsilon$ for any $i\ge0$. Let $t(\underline{a}):=\inf\{i\ge0:a_{-i-1}=2\}$ let $\bar{P}$ be the following kernel:
\begin{equation}\label{eq.discontinuous.example}
\forall\underline{a}\in\{1,2\}^{-\mathbb{N}}\,,\,\,\,\bar{P}(2|\underline{a})=p_{t(\underline{a})}\enspace. 
\end{equation}
The existence of a unique stationary chain compatible with this kernel is proven in \cite{gallo/2009} for instance. This chain is the renewal sequence, that is, a concatenation of blocks of the form $1\ldots12$ having random length with finite expectation.  It is  clearly weakly non-null, however, it is not necessarily continuous. In fact, a simple calculation shows that $\alpha_{k}=1-\sup_{l,m\ge k}|p_{l}-p_{m}|$, which needs not to go to $1$. Nevertheless, if we assume furthermore that $\sup_{k\ge0}\alpha_{k}>1-2\alpha(1)\alpha(2)$, we have $\E\paren{\prod_{k\ge0}A_{k}(U_{-k}^{-1})^{-1}}<\infty$ (see example 1 in \cite{desantis/piccioni/2010}. We now want to derive an upper bound for $\mathbb{P}(\ell(U_{-\infty}^{0})>k)$. First, define
\[
N(U_{-\infty}^{-1}):=\inf\{n\ge1: U_{-n}\in I(2)\},
\]
and observe that, $A_{k}(U_{-k}^{-1})=1$ for any $k\ge N(U_{-\infty}^{-1})$. It follows that
\begin{multline*}
\mathbb{P}(\ell(U_{-\infty}^{0})>k)=\mathbb{P}(\inf\{j\in\mathbb{N}:U_{0}<A_{j}(U_{-j}^{-1})\}>k)\\
\leq \mathbb{P}(U_{0}\geq A_{k}(U_{-k}^{-1}))\leq\mathbb{P}(A_{k}(U_{-k}^{-1})<1)\leq \mathbb{P}(N(U_{-\infty}^{-1})>k)\leq (1-\epsilon)^{k}.
\end{multline*}

\begin{obs}
The preceding theorems yield explicit upper bounds. However, they hold under restrictions we would like to surpass. 

\noindent
First, in the continuous regime, we have  assumed that $\sum_{k\ge0}(1-\alpha_{k})<+\infty$. Nevertheless, CFTP are known to exist with the weaker assumption $\sum_{k\ge1}\prod_{i=0}^{k-1}\alpha_{i}=+\infty$, and it is known that  $\bar{d}({\bf X}, {\bf X}^{[k]})$ goes to zero in this case. We will be interested in upper bounds for the rate of convergence to zero under these weak conditions. 

\noindent
Second, in Theorem \ref{coro:italia}, the assumption $\E\paren{\prod_{k\ge0}A_{k}(U_{-k}^{-1})^{-1}}<\infty$ is generally difficult to check: this is particularly clear for the example of $\bar{P}$ where it requires the (not necessary) extra-assumption $\sup_{k\ge0}\alpha_{k}>1-2\alpha(1)\alpha(2)$.  
\end{obs}
The next section will solve part of these objections.

\subsection{A simple upper bound under weak non-nullness}\label{sec:localconti}

Hereafter, we assume that $P$ is weakly non-null. Let $\Theta'[0]$ be the following subset of $\Theta[0]$:
\begin{equation}\label{eq:Theta'}
\Theta'[0]:=\{i\leq 0:\,\textrm{for any }\underline{a}\,,\,\,L(\underline{a}F_{\{i,j-1\}}(\underline{a},U_{i}^{j-1}),U_{j})\leq j-i\,,\,j=i,\ldots,0\}.
\end{equation}
We have the following theorem in which a priori nothing is assumed on the continuity.

\begin{theo}\label{lemma:1}
Assume that $P$ is weakly non-null and that we can construct a set of range partitions $\{\mathcal{I}(\underline{a})\}_{\underline{a}}$ for which $\Theta'[0]\neq\emptyset$, $\mathbb{P}$-a.s. Then, for any $\theta[0]\in \Theta'[0]$ 
\begin{equation}\label{eq:discrepancy2}
\bar{d}({\bf X},{\bf X}^{[k]})\leq \mathbb{P}(\theta[0]<-k).
\end{equation}
\end{theo}

In order to illustrate this result, let us consider the examples of continuous kernels and of the kernel $\bar{P}$. \cite{gallo/garcia/2011} proposed a unified framework, including these examples and several other cases, which provides more examples of applications of this theorem. This is postponed to Appendix \ref{sec:extgallogarcia} in order to avoid technicality. 

\vspace{0.3cm}

\paragraph{{\bf Application to the continuity regime}}
Let us first introduce the following range partition.
\begin{defi}\label{def.I1}
Let $\{\mathcal{I}^{(1)}(\underline{a})\}_{\underline{a}}$ be the range partition such that, for any $a$ and $\underline{a}$
\begin{align*}
\forall k\geq0,\quad \absj{I_{k}^{(1)}(a|a_{-k}^{-1})}&:=\alpha_{k}(a|a_{-k}^{-1})-\alpha_{k-1}(a|a_{-(k-1)}^{-1})\enspace,\\
\absj{I_{\infty}^{(1)}(a|\underline{a})}&:=P(a|\underline{a})-\lim_{k\rightarrow \infty}\alpha_{k}(a|a_{-k}^{-1})\enspace, 
\end{align*}
with the convention $\alpha_{-1}(a|\emptyset)=0$. 
\end{defi}

\noindent 
Let $F^{(1)}$ and $L^{(1)}$ be the associated functions defined in \eqref{eq:update} and \eqref{eq:length}. Let
\begin{align*}
\theta[0]&:=\max\{i\leq 0:U_{j}\leq \alpha_{j-i}\,,\,\,j=i,\ldots,0\}.
\end{align*}
Observe that $L^{(1)}$ satisfies  $L^{(1)}(\underline{a},U_{0})\leq k$ whenever $U_{0}\leq \alpha_{k}$. Hence, $\theta[0]$ belongs to $\Theta'^{(1)}[0]$, the set defined by \eqref{eq:Theta'} using $F^{(1)}$ and $L^{(1)}$. Moreover, \cite{comets/fernandez/ferrari/2002} proved that, if $\sum_{k\ge1}\prod_{i=0}^{k-1}\alpha_{i}=+\infty$ (that is, under weak non-nullnes but not necessarily summable continuity)
\begin{equation}\label{achier}
\mathbb{P}(\theta[0]<-k)\leq
v_{k}:=\sum_{j=1}^{k}\sum_{
\begin{array}{c}
t_{1},\ldots,t_{j}\geq1\\
t_{1}+\ldots+t_{j}=k
\end{array}
}\prod^{j}_{m=1}(1-\alpha_{t_{m}-1})\prod_{l=0}^{t_{m}-2}\alpha_{l}
\end{equation}
which goes to $0$. This upper bound is not very satisfactory since it is difficult to handle in general. Nevertheless,  Propositions \ref{prop:bfg} and \ref{prop:expli}, given in Appendix \ref{sec:auxi}, shed light on the behavior of this vanishing sequence. In particular, under the summable continuity assumption $\sum_{k\ge0}(1-\alpha_{k})<+\infty$, Proposition \ref{prop:bfg} states that (\ref{achier}) essentially recovers the rates of Theorems \ref{coro:expli2} and \ref{coro:italia}. Also, if there exists a constant $r\in(0,1)$ and a summable sequence $(s_{k})_{k\ge1}$ such that, $\forall k\ge1$, $1-\alpha_{k}=\frac{r}{k}+s_{k}$, then, from Proposition \ref{prop:expli}, there exists a positive constant $C$ such that
\begin{equation}\label{eq:result_weak_continuity}
\bar{d}({\bf X},{\bf X}^{[k]})\leq C\frac{(\log k)^{3+r}}{k^{2-(1+r)^{2}}}.
\end{equation}

\vspace{0.3cm}

\paragraph{{\bf Application to the kernel  $\bar{P}$}}
As a second direct application of Theorem \ref{lemma:1}, let us consider the kernel $\bar{P}$ defined in \eqref{eq.discontinuous.example}. Let $\{\mathcal{I}^{(2)}(\underline{a})\}_{\underline{a}}$ be the set of range partitions, such that $|I^{(2)}(2|\emptyset)|=\alpha(2)$, $|I^{(2)}(1|\emptyset)|=\alpha(1)$ and $I^{(2)}_{k}(a|a_{-k}^{-1})=\emptyset$ for any $k\ge1$ except for $k=t(\underline{a})+1$ for which $|I^{(2)}_{k}(1|a_{-k}^{-1})|=1-p_{k}-\alpha(1)$ and $|I^{(2)}_{k}(2|a_{-k}^{-1})|=p_{k}-\alpha(2)$. It satisfies
\[
L^{(2)}(\underline{a},U_{0})=(t(\underline{a})+1){\bf 1}\{U_{0}>\alpha_{0}\}.
\]
Hence, $\theta[0]:=\max\{i\leq 0:U_{i}\in I(2)\}$ belongs to $\Theta'[0]$ ($\Theta'[0]$ is defined by \eqref{eq:Theta'} with the functions $F^{(2)}$ and $L^{(2)}$ obtained from the set of range partitions $\{\mathcal{I}^{(2)}(\underline{a})\}_{\underline{a}}$). Therefore,
\begin{equation}\label{eq:result_renewal}
\bar{d}({\bf X},{\bf X}^{[k]})\leq \mathbb{P}(\theta[0]<-k)\leq (1-\epsilon)^{k}
\end{equation}
independently of the value $\sup_{k\ge0}\alpha_{k}$. For this simple example, Theorem \ref{lemma:1} is then less restrictive than Theorem \ref{coro:italia}.

\section{{\bf PROOFS OF THE RESULTS}}\label{sec:proofs}

\subsection{Proof of Lemma \ref{lemma:simple}}\label{sect.proof.lemma:simple}

Assume that for some $k\geq0$ we have 
\[
\sum_{i=0}^{k}|I_{i}(a|a_{-i}^{-1})|> \inf_{\underline{z}}P(a|a_{-k}^{-1}\underline{z})\enspace.
\]
Then, consider a past $\underline{z}^{\star}$ such that $|I(a)|+\sum_{i=1}^{k}|I_{i}(a|a_{-i}^{-1})|>P(a|a_{-k}^{-1}\underline{z}^{\star})$. As, for all $l\geq k+1$, $|I_{l}(a|a_{-k}^{-1}z_{-l}^{-1})|\geq0$, we have
\[
\sum_{i=0}^{k}|I_{i}(a|a_{-i}^{-1})|+\sum_{l\geq k+1}|I_{l}(a|a_{-k}^{-1}z_{-l}^{-1})|>P(a|a_{-k}^{-1}\underline{z}^{\star})\;.
\]
This is a contradiction with the second properties of the partition. This concludes the proof. \qed

\subsection{Proofs of Lemma \ref{theo:1}}
We assume that $\Theta[0]\cap\Theta^{[k]}[0]$ is $\mathbb{P}$-a.s. non-empty, and we therefore have a coupling $(\Phi({\bf U}),\Phi^{[k]}({\bf U}))$ of both chains. 
By definitions of $F^{[k]}$ and $L^{[k]}$, we observe that
when 
\begin{equation}\label{eq:mainprop}
L(\underline{b}a_{-k}^{-1},U_{0})\leq k\Rightarrow \,\textrm{ for any }\,\underline{b}\,\textrm{ we have }\left\{
\begin{array}{c}
L(\underline{b}a_{-k}^{-1},U_{0})=L^{[k]}(a_{-k}^{-1},U_{0})\,\,\textrm{and},\\
F(\underline{b}a_{-k}^{-1},U_{0})=F^{[k]}(a_{-k}^{-1},U_{0}).
\end{array}\right.
\end{equation}

\noindent
Assume that, $\forall \underline{a}\in A^{-\N}$ and  $\forall i=\theta[0],\ldots,0$, $L(\underline{a}F_{\{\theta[0],i-1\}}(\underline{a},U_{\theta[0]}^{i-1}),U_{i})\leq k$. Then, using recursively \eqref{eq:mainprop}, 
$F_{\{\theta[0],0\}}(\underline{a},U_{\theta[0]}^{0})=F^{[k]}_{\{\theta[0],0\}}(a_{-k}^{-1},U_{\theta[0]}^{0})$.  In particular, $\theta[0]\in\Theta^{[k]}[0]$ and $[\Phi({\bf U})]_{0}=[\Phi^{[k]}({\bf U})]_{0}$. Therefore,
\begin{align*}
\bar{d}({\bf X},{\bf X}^{[k]})&\leq \mathbb{P}([\Phi({\bf U})]_{0}\neq[\Phi^{[k]}({\bf U})]_{0})\\
&\leq\mathbb{P}\left(\bigcup_{\underline{a}}\bigcup_{i=\theta[0]}^{0}\left\{L\left(\underline{a}F_{\{\theta[0],i-1\}}(\underline{a},U_{\theta[0]}^{i-1}),U_{i}\right)>k\right\}\right).
\end{align*}

\subsection{Proof of Theorem \ref{coro:expli2}}\label{sect:proof.thm.expli2} This section is divided in three parts. First, as mentioned before the statement of the theorem, we define \emph{very weak non-nullness}. Then, we prove some technical lemmas allowing to apply Lemma \ref{theo:1}. Finally, we prove the theorem.
\subsubsection{Definition of \emph{very weak non-nullness}}
Consider the set of range partitions $\{\mathcal{I}^{(1)}(\underline{a})\}_{\underline{a}}$ of Definition \ref{def.I1}. 
As observed by \cite{desantis/piccioni/2010}, in the continuous case, since $\{\alpha_{k}\}_{k\ge0}$ increases monotonically to $1$, there exists $k\ge0$ such that $\alpha_{k}>0$. Let $k^{\star}$ be the smallest of these integers
and let $F^{\star}$ be the following  update function 
\[
F^{\star}(a_{-k^{\star}}^{-1},U_{0}):=F^{(1)}(\underline{b},U_{0}\alpha_{k^{\star}})\,\,\,\,\quad\forall \underline{b}\,\,\,\,\textrm{s.t.}\,\,\,a_{-k^{\star}}^{-1}=b_{-k^{\star}}^{-1}\enspace.
\] 
$F^{\star}$ is well defined, since $U_{0}\alpha_{k^{\star}}\leq \alpha_{k^{\star}}$, hence $L^{(1)}(\underline{b},U_{0}\alpha_{k^{\star}})\leq k^{\star}$. In the case where $k^\star=0$, $F^{\star}$ is simply defined as
\[
F^{\star}(\emptyset,U_{0}):=F^{(1)}(\underline{b},U_{0}\alpha_{0})=\sum_{a\in A}{\bf1}\{U_{0}\alpha_{0}\in I_{0}(a|\emptyset)\}\,\,\,\,\quad\forall \underline{b}\,\,.
\] 
\begin{defi}[Coalescence set]\label{def:coalescenceSet}For $m\ge k^{\star}+1$, let $E_{m}$, the \emph{coalescence set} (different from the set of coalescence times), be defined as the set of all $u_{-m+1}^{0}\in A^{m}$ such that
\[
F_{\{-k^{\star}+1,0\}}^{\star}\left(a_{-k^{\star}}^{-1}F^{\star}_{\{-m+1,-k^{\star}\}}\paren{a_{-k^{\star}}^{-1},\frac{u_{-m+1}^{-k^{\star}}}{\alpha_{k^{\star}}}},\frac{u_{-k^{\star}+1}^{0}}{\alpha_{k^{\star}}}\right)\,\,\,\textrm{does not depend on}\,\,\,a_{-k^{\star}}^{-1}\;.
\]
When $k^{\star}=0$ and $m=1$, we have $E_{1}:=\cup_{a\in A}I_{0}(a|\emptyset)$.
\end{defi}
\begin{defi}\label{def.vwnn}
We say that $P$ is very weakly non-null if
\begin{equation}\label{hyp.hypnn}
\exists m\ge k^\star+1\qquad \telque \qquad\P(U_{-m+1}^{0}\in E_{m})>0\enspace.
\end{equation}
\end{defi}

Weak non-nullness corresponds to $\P(U_{0}\in E_{1})>0$, hence, it implies very weak non-nullness.

\subsubsection{Technical lemmas}
Let $\Theta^{(1)}[0]$ be the set of coalescence times defined by \eqref{eq:theta} for the function $F^{(1)}$. In a first part of the proof, we define a random time $\theta[0]$ (see \eqref{eq:ahah}) and we show that it belongs to $\Theta^{(1)}[0]$ and that it has finite expectation whenever $\sum_{k\ge k^{\star}}(1-\alpha_{k})<+\infty$. This random variable is defined in the proof of Theorem 2 in \cite{desantis/piccioni/2010}.

 Recall that, by construction of the range partition $\{\mathcal{I}^{(1)}(\underline{a})\}_{\underline{a}}$, for any $\underline{a}$, $L(\underline{a},U_{i})=k$ whenever $\alpha_{k-1}\leq U_{i}<\alpha_{k}$. This means that the sequence of ranges forms a sequence $\{L_{i}\}_{i\in\mathbb{Z}}:=\{L(\underline{a},U_{i})\}_{i\in\mathbb{Z}}$ of i.i.d. $\mathbb{N}$-valued r.v.'s. We now introduce two sequences of random times in the past, which are represented on Figure \ref{fig:dessin}, in the particular case where $k^{\star}=2$.
Let
\[
W_{1}:=\sup\{m\leq 0:U_{j}< \alpha_{j-m+k^{\star}}\,,\,\,j=m,\ldots,0\},
\]
and for any $i\ge1$
\[
Y_{i}:=\inf\{m<W_{i}:U_{n}<\alpha_{k^{\star}}\,,\,\,n=m+1,\ldots,W_{i}\}
\]
and
\[
W_{i+1}:=\sup\{m\leq  Y_{i}:U_{j}< \alpha_{j-m+k^{\star}}\,,\,\,j=m,\ldots,Y_{i}\}.
\]

\begin{figure}[htp]
\centering
\includegraphics{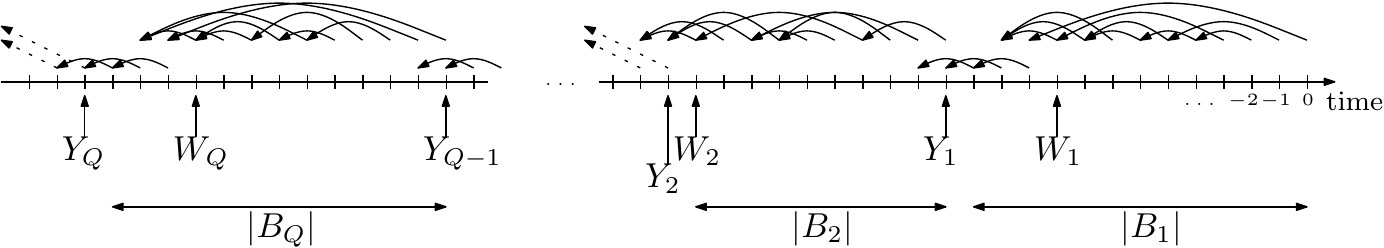}
\caption{We consider a realization of $L_{-\infty}^{0}$ in the particular case $k^{\star}=2$, that is, the arrows, which represent the length function at each time index, have length larger or equal to $2$.}
\label{fig:dessin}
\end{figure}

\noindent
Consider now the random variable 
\[
Q:=\inf\{i\ge1:(U_{Y_{i}+1},\ldots,U_{W_{i}-1})\in E_{W_{i}-Y_{i}-1}\}
\]
(see Definition \ref{def:coalescenceSet} for $E_{m}$) and put 
\begin{equation}\label{eq:ahah}
\theta[0]:=Y_{Q}.
\end{equation}

\begin{lemma}\label{lem1}
$\theta[0]\in\Theta^{(1)}[0]$.
\end{lemma}
\begin{proof}
If $\theta[0]=-k$, then there exists some $l(=-W_{Q}+1)\leq k$ such that $U_{-i}\leq \alpha_{k^{\star}}$, $i=l,\ldots,k$, and moreover, $U_{-k}^{-l}\in E_{k-l+1}$, that is
\[
F_{\{-l-k^{\star}+1,-l\}}^{\star}\left(a_{-k^{\star}}^{-1}F^{\star}_{\{-k,-l-k^{\star}\}}(a_{-k^{\star}}^{-1},\frac{1}{\alpha_{k^{\star}}}U_{-k}^{-l-k^{\star}}),\frac{1}{\alpha_{k^{\star}}}U_{-l-k^{\star}+1}^{-l}\right)\,\,\,\,\,\textrm{is independent of }\,a_{-k^{\star}}^{-1}.
\]
Since $U_{-i}\leq \alpha_{k^{\star}}$, $i=l,\ldots,k$, it follows that 
\[
F_{\{-l-k^{\star}+1,-l\}}\left(\underline{b}F_{\{-k,-l-k^{\star}\}}(\underline{b},U_{-k}^{-l-k^{\star}}),U_{-l-k^{\star}+1}^{-l}\right)\,\,\,\,\,\textrm{is independent of }\,\underline{b}.
\]
By definition of the random times $W_{i}$, all the symbols in times $\set{W_{Q},\ldots,0}$ can then be built using those in times $\set{W_{Q}-k^{\star},\ldots,W_{Q}-1}$ since none of the arrows from time $W_{Q}$ until $0$ go further time $W_{Q}-k^{\star}$, see the Figure \ref{fig:dessin}. Therefore, the construction of the symbol at times $0$ does not depend on the symbols before $\theta[0]$, i.e $\theta[0]\in\Theta^{(1)}[0]$.
\end{proof}

\begin{lemma}\label{lemma:finiteexpect}
$\mathbb{E}|W_{1}|<+\infty$ whenever $\sum_{k\ge k^{\star}}(1-\alpha_{k})<+\infty$.
\end{lemma}

\begin{proof}
Letting $\bar{\alpha}_{l-k^{\star}}=\alpha_{l}$ for any $l\ge k^{\star}$, we have
\[
W_{1}=\sup\{m\leq 0:U_{j}< \bar{\alpha}_{j-m}\,,\,\,j=m,\ldots,0\}.
\]
Thus $W_{1}$ is defined exactly as $\tau[0]$ of display (4.2) in \cite{comets/fernandez/ferrari/2002}, substituting their $a_{k}$'s by our $\bar{\alpha}_{k}$'s. They proved (see display (4.6) and item (ii) Proposition 5.1 therein) that $\mathbb{E}|\tau[0]|<+\infty$ whenever $\sum_{k\ge 0}(1-\bar{\alpha}_{k})<+\infty$. It follows that $\mathbb{E}|W_{1}|<+\infty$ whenever $\sum_{k\ge k^{\star}}(1-\alpha_{k})<+\infty$.
\end{proof}

\begin{lemma}\label{lem2}
$\mathbb{E}|\theta[0]|<+\infty$ whenever $\sum_{k\ge k^{\star}}(1-\alpha_{k})<+\infty$.
\end{lemma}
\begin{proof}As observed in \cite{desantis/piccioni/2010}, $\{W_{i}-Y_{i}-1\}_{i\ge1}$ is a sequence of i.i.d. geometric r.v.'s with success probability $1-\alpha_{k^{\star}}$ and $\{Y_{i}-W_{i+1}\}_{i\ge0}$ (with $Y_{0}:=0$) is a sequence of i.i.d. r.v.'s distributed as $-W_{1}$, conditional to be non-zero. Moreover, Lemma \ref{lemma:finiteexpect} states that $\mathbb{E}|W_{1}|<+\infty$.
It follows that $\{B_{i}\}_{i\ge1}:=\{Y_{i-1}-Y_{i}-1\}_{i\ge1}$ is  a sequence of i.i.d. $\mathbb{N}$-valued r.v.'s with finite expectation. 
Thus
\(
\sum_{k=1}^{n}B_{i}-n\mathbb{E}B_{1}
\)
forms a martingale with respect to the filtration $\mathcal{F}(B_{1},\ldots,B_{i}\,:\,\,i\ge1)$ and we have by the optional sampling theorem 
\[
\mathbb{E}|\theta[0]|:=\mathbb{E}|Y_{Q}|=\mathbb{E}\left(\sum_{i=1}^{Q}B_{i}\right)=\mathbb{E}Q.\mathbb{E}B_{1}<+\infty.
\]
\end{proof}

We finally need the following lemma.

\begin{lemma}\label{lem3}
For any $k\ge k^{\star}$, $\theta[0]\in\Theta^{(1),[k]}[0]$.
\end{lemma}
\begin{proof}
For any $k\ge k^{\star}$, $F^{(1),[k]}$ and $L^{(1),[k]}$ satisfy \eqref{eq:mainprop}. This implies that, in the interval $\{Y_{Q},\ldots,W_{Q}-1\}$,  coalescence occurs as well for $F^{(1),[k]}$, i.e. $U_{-\theta[0]}^{W_{Q}-1}\in E^{[k]}_{\theta[0]-W_{Q}}$. Both constructed chains are equals until the first time $F^{(1)}$ uses a range larger than $k$. But at this moment, due to the definition of the $W_{i}$'s, we have already perfectly simulated at least $k$ symbols of both chains, and therefore, we can continue constructing until time $0$ because the ranges of $F^{(1),[k]}$ are smaller of equal to $k$.  It follows that $Y_{Q}$ is a coalescence time for $F^{(1),[k]}$, and therefore, $\theta[0]\in\Theta^{(1),[k]}[0]$ for any $k\ge k^{\star}$.
\end{proof}

\subsubsection{Proof of Theorem \ref{coro:expli2}} By definition of $F^{(1)}$ and $L^{(1)}$, we have for any sufficiently large $k$, 
\[
\bigcup_{\underline{a}}\bigcup_{i=\theta[0]}^{0}\left\{L^{(1)}\left(\underline{a}F^{(1)}_{\{\theta[0],i-1\}}(\underline{a},U_{\theta[0]}^{i-1}),U_{i}\right)>k\right\}\subset\bigcup_{i=\theta[0]}^{0}\{U_{i}>\alpha_{k}\}.
\]
By Lemmas \ref{lem1} and \ref{lem3}, Lemma \ref{theo:1} applies and gives, for sufficiently large $k$'s
\begin{align*}
\bar{d}({\bf X},{\bf X}^{[k]})\leq \mathbb{P}\left(\bigcup_{i=\theta[0]}^{0}\{U_{i}>\alpha_{k}\}\right)=\mathbb{P}\left(\sum_{i=0}^{|\theta[0]|}{\bf 1}\{U_{i}>\alpha_{k}\}\ge1\right)\leq \mathbb{E}\left(\sum_{i=0}^{|\theta[0]|}{\bf 1}\{U_{i}>\alpha_{k}\}\right)
\end{align*}
where we used the Markov inequality for the last inequality.
Using the fact that $\theta[0]$ is a stopping time in the past for the sequence $U_{i}, i\leq 0$, and that it has finite expectation by Lemma \ref{lem2}, we can apply the Wald's equality to obtain 
\[
\bar{d}({\bf X},{\bf X}^{[k]})\leq \mathbb{E}|\theta[0]|.\mathbb{E}{\bf 1}\{U_{i}>\alpha_{k}\}=\mathbb{E}|\theta[0]|.(1-\alpha_{k}).
\]

\subsection{{\bf Proof of Theorem \ref{coro:italia}}} We divide this proof into two parts. First, we prove technical lemmas allowing to use Lemma~\ref{theo:1}. Then, we prove the theorem.

\subsubsection{Technical lemmas}
Using the quantity $\ell(U_{-\infty}^{0})$ defined  by \eqref{eq:ell}, we define
%
\begin{equation}\label{eq:theta0}
\theta[0]:=\sup\{j\leq 0:\ell(U_{-\infty}^{i})\leq i-j\,,\,\,i=j,\ldots,0\}.
\end{equation}
\begin{lemma}\label{lemma:ahahah}
$\theta[0]$, defined by \eqref{eq:theta0}, belongs to $\Theta^{(1)}[0]\cap \Theta^{(1),[k]}[0]$  for any $k\geq 0$ and 
\begin{equation}\label{eq:italia}
 \mathbb{P}\left(\bigcup_{\underline{a}}\bigcup_{i=\theta[0]}^{0}\left\{L^{(1)}\left(\underline{a}F^{(1)}_{\{\theta[0],i-1\}}(\underline{a},U_{\theta[0]}^{i-1}),U_{i}\right)>k\right\}\right)\leq\mathbb{P}\left(\bigcup_{i=\theta[0]}^{0}\left\{\ell\left(U_{-\infty}^{i}\right)>k\right\}\right).
\end{equation}
\end{lemma}

\begin{proof}

For any $U_{-k}^{-1}$, the way the sets of strings $\{\underline{z}F^{(1)}_{\{-k,-1\}}(\underline{z},U_{-k}^{-1})\}_{\underline{z}}$ and $J_{k}(U_{-k}^{-1})$ are defined ensure that the former is included in the later. 
It follows  that, for any $U_{-k}^{-1}$,
\[
A_{k}(U_{-k}^{-1}):=\inf_{x_{-k}^{-1}:\underline{x}\in J_{k}(U_{-k}^{-1})}\sum_{a\in A}\inf_{\underline{z}}P(a|x_{-k}^{-1}\underline{z})\leq \sum_{a\in A}\inf_{\underline{z}}P(a|F^{(1)}_{\{-k,-1\}}(\underline{z},U_{-k}^{-1})\underline{z}).
\]
As the inequality $A_{0}\leq \sum_{a\in A}\inf_{\underline{z}}P(a|\underline{z})$ is also true, we deduce that, for any $k\ge0$, any $U_{-\infty}^{0}\in [0,1[^{-\mathbb{N}}$ and any $\underline{z}\in A^{-\mathbb{N}}$,
\[
\ell(U_{-\infty}^{0})\leq k\Rightarrow L^{(1)}\left(\underline{z}F^{(1)}_{\{-k,-1\}}(\underline{z},U_{-k}^{-1}),U_{0}\right)\le k\enspace.
\]
By recurrence, this means that, for all $\theta[0]\le i\le 0$, $F^{(1)}_{\{\theta[0],i\}}(\underline{z},U_{\theta[0]}^{i})$ does not depend on $\underline{z}$. Hence, $\theta[0]$ is also a coalescence time for the update function $F^{(1)}$, that is $\theta[0]\in\Theta^{(1)}[0]$.  Observe that we have proved, more specifically, that $\theta[0]\in\Theta'^{(1)}[0]$, where $\Theta'^{(1)}[0]\subset\Theta^{(1)}[0]$ is defined by \ref{eq:Theta'} using $F^{(1)}$ and $L^{(1)}$. By Lemma \ref{lemma:weakK} below, this implies that $\theta[0]\in \Theta^{(1),[k]}[0]$ for any $k\geq 0$ as well.

\vspace{0.2cm}

\noindent We now prove the second statement of the lemma. If there exist $i\ge k$, $U_{-i}^{-1}$ and $\underline{z}$ such that $L^{(1)}\left(\underline{z}F^{(1)}_{\{-i,-1\}}(\underline{z},U_{-i}^{-1}),U_{0}\right)> k$, then there exists some past $\underline{a}$ (take $\underline{a}=\underline{z}F^{(1)}_{\{-i,-k-1\}}(\underline{z},U_{-i}^{-k-1})$ for instance) such that $L^{(1)}\left(\underline{a}F^{(1)}_{\{-k,-1\}}(\underline{a},U_{-k}^{-1}),U_{0}\right)> k$.

We now have the following sequence of inclusions
\begin{align*}
\bigcup_{\underline{a}}\bigcup_{i=\theta[0]}^{0}&\left\{L^{(1)}\left(\underline{a}F^{(1)}_{\{\theta[0],i-1\}}(\underline{a},U_{\theta[0]}^{i-1}),U_{i}\right)>k\right\}\\
&=\bigcup_{\underline{a}}\bigcup_{i=\theta[0]}^{0}\left\{L^{(1)}\left(\underline{a}F^{(1)}_{\{\theta[0],i-1\}}(\underline{a},U_{\theta[0]}^{i-1}),U_{i}\right)>k\right\}\cap\{\theta[0]\leq i-k-1\}\\
&\subset\bigcup_{\underline{a}}\bigcup_{i=\theta[0]}^{0}\left\{L^{(1)}\left(\underline{a}F^{(1)}_{\{i-k,i-1\}}(\underline{a},U_{i-k}^{i-1}),U_{i}\right)>k\right\}\cap\{\theta[0]\leq i-k-1\}\\
&\subset\bigcup_{i=\theta[0]}^{0}\left\{\ell\left(U_{-\infty}^{i}\right)>k\right\}.
\end{align*}
This concludes the proof of the lemma.
\end{proof}

Recall the definition \eqref{eq:Theta'} of $\Theta'[0]$ for generic range partitions of a weakly non-null kernel $P$. We will need the following lemma.

\begin{lemma}\label{lemma:weakK}
For any $k\ge0$, $\Theta'[0]\subset\Theta^{[k]}[0]$.
\end{lemma}
\begin{proof}
Let $\theta[0]\in\Theta'[0]$. For any fixed $k\ge0$, we separate two cases.
\begin{enumerate}
\item If $\theta[0]\ge -k$, then, by the definition of $\Theta'[0]$, the ranges used by $F$ from $\theta[0]$ to $0$ are all smaller than or equals to $k$, and therefore using \eqref{eq:mainprop}, we have that the length used by $F^{[k]}$ in the same interval of indexes are the same and the constructed symbols are the same as well. Thus $\theta[0]\in\Theta^{[k]}[0]$.
\item If $\theta[0]<-k$, then, by the definition of $\Theta'[0]$, we can apply the same method as in the preceding case, and obtain that $\theta[0]$ is a coalescence time for $F^{[k]}$ for the time indexes from $\theta[0]$ up to $\theta[0]+k$. But $\theta[0]$ is also a coalescence time for the time indexes from $\theta[0]+k+1$ up to $0$, since the ranges used by $F^{[k]}$ are always smaller than or equal to $k$. Thus, in this case also, $\theta[0]\in\Theta^{[k]}[0]$.
\end{enumerate}
\end{proof}
\subsubsection{Proof of Theorem \ref{coro:italia}}
In the conditions of this theorem, by Theorem 1 in \cite{desantis/piccioni/2010}, $\theta[0]$ is $\mathbb{P}$-a.s. finite. Moreover, by Lemma \ref{lemma:ahahah}, $\theta[0]\in \Theta^{(1)}[0]\cap\Theta^{(1),[k]}[0]$ for any $k\ge0$. Thus we can apply Lemma \ref{theo:1}, and obtain, using  Lemma \ref{lemma:ahahah}
\begin{align*}
\bar{d}({\bf X},{\bf X}^{[k]})\leq \mathbb{P}\left(\bigcup_{i=\theta[0]}^{0}\left\{\ell\left(U_{-\infty}^{i}\right)>k\right\}\right)
\end{align*}
and moreover
\begin{align*}
\mathbb{P}\left(\bigcup_{i=\theta[0]}^{0}\left\{\ell\left(U_{-\infty}^{i}\right)>k\right\}\right)&=\mathbb{P}\left(\sum_{i=\theta[0]}^{0}{\bf 1}\left\{\ell\left(U_{-\infty}^{i}\right)>k\right\}\ge1\right)\\
&\leq \E\paren{\sum_{i=\theta[0]}^{0}{\bf 1}\left\{\ell\left(U_{-\infty}^{i}\right)>k\right\}}\enspace.
\end{align*}
Consider the $\sigma$-algebra $\Fcal_{k}$ generated by $U_{-k}^{0}, \,k\ge0$. Then, $\ell(U_{-\infty}^{0})$ is a stopping time with respect to $\Fcal_{k}$ and, by definition, so is ${\theta}[0]$. Moreover, $\ell(U_{-\infty}^{i})$ is independent of $U_{i+1}^{0}$ by independence of the $U_{j}$'s. Finally, by stationarity, $\ell\left(U_{-\infty}^{i}\right)\stackrel{\mathcal{D}}{=}\ell\left(U_{-\infty}^{0}\right)$, hence $\E\paren{{\bf 1}\left\{\ell\left(U_{-\infty}^{i}\right)>k\right\}}=\E\paren{{\bf 1}\left\{\ell\left(U_{-\infty}^{0}\right)>k\right\}}$, for any $i\in\Z$. By Theorem 1 in \cite{desantis/piccioni/2010}, $\theta[0]$ has finite expectation, hence we can use Wald equality to obtain
\begin{equation}\label{eq:contEtheta0finite}
\bar{d}({\bf X},{\bf X}^{[k]})\leq \E|\theta[0]|\mathbb{P}(\ell\left(U_{-\infty}^{0}\right)>k)\enspace. 
\end{equation}
This concludes the proof of Theorem \ref{coro:italia}.

\subsection{{\bf Proof of Theorem \ref{lemma:1}}}
Recall the definition of the set $\Theta'[0]$ given by \eqref{eq:Theta'}. If $\theta[0]\in\Theta'[0]$ and ${\theta}[0]\ge-k$, then we are sure that the range $L\left(\underline{a}F_{\{\theta[0],i-1\}}(\underline{a},U_{\theta[0]}^{i-1}),U_{i}\right)\leq k$ for any $i={\theta}[0],\ldots,0$ and any $\underline{a}$, therefore
\[
\bigcup_{\underline{a}}\bigcup_{i=\theta[0]}^{0}\left\{L\left(\underline{a}F_{\{\theta[0],i-1\}}(\underline{a},U_{\theta[0]}^{i-1}),U_{i}\right)>k\right\}\subset\{{\theta}[0]<-k\}.
\]
By Lemma \ref{lemma:weakK}, any $\theta[0]\in \Theta'[0]$ also belongs to $\Theta^{[k]}[0]$ for any $k\ge0$. We can thus apply Lemma \ref{theo:1} and conclude the proof of the theorem.
\bibliographystyle{jtbnew}
\bibliography{sandro_bibli}

\appendix
\section{Local continuity with respect to the past $\underline{1}$}\label{sec:extgallogarcia}

In this section, we assume that $A=\{1,2\}$, and that $P$ has only one discontinuity point, the point $\underline{1}=\ldots111$. We refer the interested reader to \cite{gallo/garcia/2011} for examples with countable alphabets, and discontinuities in more complicated set of pasts.  To begin, we need the following definition.

\begin{defi}[Local continuity with respect to the past $\underline{1}$]
We say that a kernel $P$ on $\{1,2\}$ is \emph{locally continuous with respect to the past $\underline{1}$} if
\[
\forall i\geq 0 ,\qquad \inf_{a_{-k}^{-1}}\sum_{a\in A}\inf_{\underline{z}}P(a|1^{i}2a_{-k}^{-1}\underline{z})
\]
converges to $1$ as $k$ diverges. We distinguish two particular situations of interest.
\begin{itemize}
\item We say that $P$ is \emph{strongly locally continuous with respect to $\underline{1}$} if there exists an integer function $\ell:\mathbb{N}\rightarrow\mathbb{N}$ such that 
\begin{equation}\label{eq:stronglocal}
\forall i\geq 0,\qquad \inf_{a_{-k}^{-1}}\sum_{a\in A}\inf_{\underline{z}}P(a|1^{i}2a_{-k}^{-1}\underline{z})=1
\end{equation}
for any $k\ge \ell(i)$, and
\item we say that $P$ is \emph{uniformly locally continuous with respect to $\underline{1}$} if
\begin{equation}\label{eq:uniflocal}
\alpha^{\underline{1}}_{k}:=\inf_{i\geq0}\inf_{a_{-k}^{-1}}\sum_{a\in A}\inf_{\underline{z}}P(a|1^{i}2a_{-k}^{-1}\underline{z})
\end{equation}
converges to $1$ as $k$ diverges.
\end{itemize}
\end{defi}

Strongly locally continuous kernels are known as \emph{probabilistic context trees}, a model that have been introduced by \cite{rissanen/1983} as a \emph{universal data compression model}. It was first consider, from the ``CFTP point of view'', by \cite{gallo/2009}. 
The kernel $\bar{P}$ is a simple example which is strongly and uniformly locally continuous with respect to $\underline{1}$. 

\vspace{0.2cm}

\paragraph{{\bf Assumption 1}:} $P$ is strongly locally continuous with respect to $\underline{1}$.

\vspace{0.2cm}

\paragraph{{\bf Assumption 2}:} $P$ is uniformly locally continuous with respect to $\underline{1}$.

\begin{notation}\label{nota}
Let us introduce the following notation.
\begin{itemize}
\item Stationary chains compatible with kernels satisfying Assumptions i=1 and 2 are denoted ${\bf X}^{(i)}$, and the corresponding canonical $k$-steps Markov approximations are denoted ${\bf X}^{(i),[k]}$. 
\item We use the notations $r^{(i)}_{0}:=\alpha_{0}$ for i=1 and 2, and for $k\ge1$,
\begin{align*}
r^{(1)}_{k}&:= r^{(1)}_{k-1}\vee (1-(1-\alpha(2))^{\ell^{-1}(k)})\\
r^{(2)}_{k}&:=r^{(2)}_{k-1}\vee (1-(1-\alpha^{\underline{1}}_{k})/\alpha(2))
\end{align*}
where  $\ell$ and $\alpha^{\underline{1}}_{k}$ are the parameters of the kernels under assumptions 1 and 2 respectively.
\item For i=1 and 2
\begin{equation}\label{eq.bound.HC}
v_{k}^{(i)}:=\sum_{j=1}^{k}\sum_{
\begin{array}{c}
t_{1},\ldots,t_{j}\geq1\\
t_{1}+\ldots+t_{j}=k
\end{array}
}\prod^{j}_{m=1}(1-r^{(i)}_{t_{m}-1})\prod_{l=0}^{t_{m}-2}r^{(i)}_{l}
\end{equation}
where $\prod_{l=0}^{-1}:=1$.
\item And finally, for any $k\ge0$, let 
\begin{equation}\label{eq:nota}
u_{k}:= \lfloor k\alpha(2)/2\rfloor\mathbb{P}\left(\left|\sum_{j=0}^{\lfloor k\alpha(2)/2\rfloor}\xi_{j}-\frac{\lfloor k\alpha(2)/2\rfloor}{\alpha(2)}\right|>k/2\right)
\end{equation}
It is well-known that this sequence goes exponentially fast to $0$ (see \cite{kallenberg/2002} for instance). An explicit upper bound is derived in Appendix \ref{sec.conc.geom}.
\end{itemize}
\end{notation}

\begin{coro}\label{coro:nonexpli}
Under the weak non-nullness assumption, we have for i=1 and 2 that, if $\sum_{k\ge1}\prod_{i=0}^{k-1}r^{(i)}_{k}=\infty$, 
\begin{equation}\label{eq:expli1}
\bar{d}({\bf X}^{(i)},{\bf X}^{(i),[k]})\leq u_{k}+v_{\lfloor k\alpha(2)/2\rfloor}^{(i)}\rightarrow0.
\end{equation}
\end{coro}

%
The quantity defined on display \eqref{eq.bound.HC} is related to the house of card process presented in Section \ref{sec:auxi} (see equation \eqref{bound2}). We provide in Propositions \ref{prop:bfg} and \ref{prop:expli} explicit upper-bounds on the term \eqref{eq.bound.HC} that can be plugged in \eqref{eq:expli1}. The term \eqref{eq:nota} is studied in Corollary \ref{coro.contuk}.
It follows  in particular from these propositions that, whenever $r_{k}^{(i)}$ is not exponentially decreasing, the leading term in \eqref{eq:expli1} is $v_{k}^{(i)}$, and therefore, we obtain for some constant $C>1$ and any sufficiently large $k$
\[
\bar{d}({\bf X}^{(i)},{\bf X}^{(i),[k]})\leq Cv_{\lfloor k\alpha(2)/2\rfloor}^{(i)}.
\]
For instance, Proposition \ref{prop:expli}, states that, if $1-r^{(i)}_{k}=\frac{r}{k}+s_{k}$, $k\ge1$ with $r\in(0,1)$ and  $\{s_{k}\}_{k\ge1}$ is any summable sequence, we obtain for some constant $C>1$
\begin{equation}\label{eq:result_weak_continuity}
\bar{d}({\bf X}^{(i)},{\bf X}^{(i),[k]})\leq C\frac{(\log k)^{3+r}}{k^{2-(1+r)^{2}}}.
\end{equation}

\begin{proof}
Under Assumptions 1 and 2 with weak non-nullness, \cite{gallo/garcia/2011} constructed a set of range partitions generating a set of coalescence times $\Theta'[0]$ which is a.s. non-empty. This is what is stated in  Corollaries 6.1 and 6.2 (and the discussions following them) for respectively  Assumption 1 and 2. 
They defined a random time $\Lambda^{(i)}[0]$ (see display (34) therein) which belongs to $\Theta'[0]$, as stated by Lemma 8.1 therein. They also prove that $\mathbb{P}(\Lambda^{(i)}[0]<-k)$ is upper bounded by $u^{(i)}_{k}+v^{(i)}_{\lfloor k\alpha(2)/2\rfloor}$ (where $\{u_{k}^{(i)}\}_{k\ge1}$ has been defined by (\ref{eq:nota})). This is in fact stated in the proof of item (ii) of Theorem 5.2 therein.

By Theorem \ref{lemma:1}, these upper bounds are therefore upper bounds for the $\bar{d}$-distance $\bar{d}({\bf X},{\bf X}^{[k]})$.

\end{proof}

\section{Some results on the House of Cards Markov chain}\label{sec:auxi}

Fix a non-decreasing sequence $\{r_{k}\}_{k\ge0}$  of $[0,1]$-valued real numbers converging  to $1$. 
The house of Cards Markov chain ${\bf H}=\{H_{n}\}_{n\ge0}$ related to this sequence is the $\mathbb{N}$-valued Markov chain starting from state $0$ and having transition matrix $Q=\{Q(i,j)\}_{i\ge0,\,j\ge0}$ where $Q(i,j):=r_{i}{\bf 1}\{j=i+1\}+(1-r_{i}){\bf 1}\{j=0\}$.  Let us denote $v_{k}:=\Pr(H_{k}=0)$, the probability that the house of cards is at state $0$ at time $k$. We want to derive explicit rates of convergence to $0$ of this sequence when ${\bf H}$ is not positive recurrent.  These results will be used in the next section in order to obtain explicit upper bounds for $\bar{d}({\bf X},{\bf X}^{[k]})$ under several types of assumptions. Decomposing the event $\set{H_{k}=0}$ into the possible come back of the process $\{H_{\ell}\}_{\ell=0,\ldots,k}$ to $0$ yields, for any $n\ge 1$
\begin{equation}\label{bound2}
v_{k}:=\sum_{j=1}^{k}\sum_{
\begin{array}{c}
t_{1},\ldots,t_{j}\geq1\\
t_{1}+\ldots+t_{j}=k
\end{array}
}\prod^{j}_{m=1}(1-r_{t_{m}-1})\prod_{l=0}^{t_{m}-2}r_{l},
\end{equation}
where $\prod_{l=0}^{-1}:=1$. Although explicit, this bound cannot be used directly and has to be simplified.
%
As a first insight, we borrow the following Proposition of \cite{bressaud/fernandez/galves/1999a}.

\begin{prop}\label{prop:bfg}
\begin{enumerate}[(i)]
\item $v_{k}$ goes to zero as $k$ diverges if $\sum_{m\geq 1}\prod_{l=0}^{m-1}r_{l}=+\infty$,
\item $v_{k}$ is summable in $k$ if $1-r_{k}$  is summable in $k$,
\item $v_{k}$ behaves as $O(1-r_{k})$ if $1-r_{k}$  is summable in $k$ and $\sup_{j}\limsup_{k\rightarrow+\infty}(\frac{1-r_{j}}{1-t_{kj}})\leq1$
\item $v_{k}$ goes to zero exponentially fast if $1-r_{k}$ decreases exponentially.
\end{enumerate}
\end{prop}

As observe in \cite{bressaud/fernandez/galves/1999a}, the conditions of item (iii) are satisfied if, for example, $1-r^{(i)}_{k}\sim (\log k)^{\eta}k^{-\zeta}$ for some $\zeta>1$, and for any $\eta$. However, this is one of the only cases in which this proposition yields explicit rates. 
In the present paper, we will prove the following proposition.

\begin{prop}\label{prop:expli}
We have the following explicit upper bounds.
\begin{enumerate}[(i)]
\item A non summable case: if $1-r_{k}= \frac{r}{k}+s_{k},\,k\ge1$ where $r\in(0,1)$ and $\{s_{n}\}_{n\ge1}$ is a summable sequence, there exists a constant $C>0$ such that
\[
v_{k}\leq C\frac{(\ln k)^{3+r}}{(k)^{2-(1+r)^{2}}}.
\]
\item Generic summable case: if $t_{\infty}:=\prod_{k\ge0}r_{k}>0$, then
\[
v_{k}\leq \inf_{K=1,\ldots,k}\set{K^{2}(1-r_{k/K})+(1-t_{\infty})^{K}}.
\]
\item Exponential case: if $1-r_{k}\leq C_{r}r^{k},\,k\ge1$, for some $r\in(0,1)$ and a constant $C_{r}\in(0,\log \frac{1}{r})$ then
\[
v_{k}\leq \frac1{C_{r}}(e^{C_{r}}r)^{k}.
\]
\end{enumerate}
\end{prop}

\subsection{Proof of Proposition \ref{prop:expli}}
Before we come into the proofs of each item of this proposition, let us collect some simple remarks on the House of Cards Markov chain.

Let $\{T_{k}\}_{k\ge0}$ be a sequence of the stopping times defined as
$T_{0}\egaldef 0$ and, recursively, for any $k\geq 1$, $T_{k}\egaldef\inf\set{l\geq T_{k-1}+1\telque H_{l}=0}$. The Markov property ensures that 
the random variables $I_{k}\egaldef T_{k+1}-T_{k}$ are i.i.d., valued in $\N^{*}$ and it is easy to check that
\[
\forall k\geq 1,\qquad t_{k}\egaldef \Qr{I_{1}=k}=(1-r_{k-1})\prod_{i=0}^{k-2}r_{i}\enspace,
\]
 where $\prod_{l=0}^{-1}:=1$. We have, for any $n\geq 0$,
\[
\Qr{H_{n}=0}=\Qr{\exists k\geq 0,\telque T_{k}=n}=\sum_{k=0}^{\infty}\Qr{T_{k}=n}\enspace.
\]
We write $T_{k}=\sum_{l=0}^{k-1}I_{l}$. As all the $I_{l}\geq 1$, we have $\Qr{T_{k}=n}=0$ for all $k> n$. Therefore, for all $K\in[1,n]$,
\begin{equation}\label{eq.alpha.summable}
\Qr{H_{n}=0}=\sum_{k=0}^{n}\Qr{T_{k}=n}=\sum_{k=0}^{K}\Qr{T_{k}=n}+\sum_{k=K+1}^{n}\Qr{T_{k}=n}\enspace.
\end{equation}

\begin{fact}\label{fact.F1}
Let $K\in[1,n]$, we have $\Qr{\forall l\in[1,K],\; I_{l}\leq n}=(1-\nu_{n+1})^{K}$. In particular, if $K
\in[1,n]$, then
\begin{align*}
\Qr{\exists j\in[K,n],\telque \sum_{l=0}^{j}I_{l}=n}&\leq \Qr{\forall l\in[1,K], I_{l}\leq n}\\
&=(1-\nu_{n+1})^{K}\enspace.
\end{align*}
\end{fact}
In order to control $\sum_{k=0}^{K}\Qr{T_{k}=n}=\Qr{\exists k=0,\ldots,K,\;T_{k}=n}$, we can simply remark that, if there exists $k\in1,\ldots K$ such that $\sum_{i=1}^{k}I_{l}=n$, there exists necessarily $i\in[1,K]$ and $r\in[1,\ldots,K]$ such that $I_{i}=n/r$. This implies that 
\begin{align*}
\Qr{\exists k=0,\ldots,K,\;T_{k}=n}&\leq \Qr{\exists i\in[1,K],\;\exists r\in[1,\ldots,K],\telque I_{i}=\frac nr}\\
&\leq \sum_{i=1}^{K}\sum_{j=1}^{K}\Qr{I_{1}=\frac nr}\leq K^{2}t_{n/K}\enspace.
\end{align*}
We have obtain the following result.
\begin{fact}\label{fact.F2}
Let $K\in[1,n]$, we have 
\[
\sum_{k=0}^{K}\Qr{T_{k}=n}\leq K^{2}t_{n/K}\enspace.
\]
\end{fact}

Restricting our attention to the summable case (that is, when $\sum_{k\ge0}(1-r_{k})<+\infty$), the following fact is fundamental. Its proof is immediate.
\begin{fact}\label{fact.F3}
If $\sum_{n\ge0}(1-r_{n})<\infty$, then $t_{\infty}\egaldef \Qr{I_{1}=\infty}=\prod_{i=0}^{\infty}r_{i}>0$, in particular, $\nu_{n}\egaldef 
\Qr{I_{1}\geq n}\geq t_{\infty}>0$. Moreover, for all $n\in\N$, $t_{\infty}(1-r_{n})\leq t_{n}\leq (1-r_{n})$
\end{fact}

Using Facts \ref{fact.F1}, \ref{fact.F2} and \ref{fact.F3}, we are ready to prove items (i) and (ii) of Proposition \ref{prop:expli}.

\begin{proof}[Proof of Item (i) of Proposition \ref{prop:expli}]
As far as we know, all the results on the house of card process hold in the summable case. When $\sum_{k\in\N}(1-r_{k})=\infty$, it is only known that $\sum_{n\in \N}\Qr{H_{n}=0}=\infty$. It is interesting to notice that we can still obtain some rate of convergence for $\Qr{H_{n}=0}$ from our elementary facts, at least in the following example. Let us assume that there exists $r<1$ and a summable sequence $s_{n}$ such that, for all $n\geq 1$, $1-r_{n}=\frac{r}n+s_{n}$. In this case, we have $\sum_{n\in \N}(1-r_{n})=\infty$, therefore $t_{\infty}=0$. Nevertheless, 
\[
\prod_{i=0}^{n}r_{i}\leq \prod_{i=1}^{n}e^{-(1-r_{i})}=e^{-r \ln n+O(1)}\leq Cn^{-r}\enspace.
\]
Therefore $t_{n}\leq Cn^{-(1+r)}$. Moreover, using the inequality $(1-u)\geq e^{-u-u^{2}}$, valid for all $u<1/8$, we see that $t_{n}\geq cn^{-(1+r)}$. Therefore, $\nu_{n}=\sum_{k\geq n}t_{k}\geq cn^{-r}$. It follows from Fact \ref{fact.F1} that, for large $K$ and $n$,
\[
\sum_{k=K+1}^{n}\Qr{T_{k}=n}\leq (1-\nu_{n+1})^{K}\leq e^{-cKn^{-r}}\enspace.
\]
Using Fact \ref{fact.F1}, we also have
\begin{equation}\label{eq.petitKsimple}
\sum_{k=0}^{K}\Qr{T_{k}=n}\leq CK^{2}t_{n/K}\leq CK^{3+r}n^{-(1+r)}\enspace.
\end{equation}
We deduce then from \eqref{eq.alpha.summable} that, for all $K\in [0,n]$,
\[
\Qr{H_{n}=0}\leq C\paren{\frac{K^{3+r}}{n^{1+r}}+e^{-cKn^{-r}}}\enspace.
\]
For $K=2n^{r}\ln n$, we obtain
\[
\Qr{H_{n}=0}\leq C\frac{(\ln n)^{3+r}}{n^{1-2r-r^{2}}}=C\frac{(\ln n)^{3+r}}{n^{2-(1+r)^{2}}}\enspace.
\]
If $0<r<1$, we have $2-(1+r)^{2}>0$. This bound may not be optimal, but it is interesting to see that we still can derive rates of convergence from our basic remarks even in this pathological example. 
\end{proof}

%
\begin{proof}[Proof of Item (ii) of Proposition \ref{prop:expli}]
We deduce from Facts \ref{fact.F1} and \ref{fact.F3} that, in the summable case 
\[
\sum_{k=K+1}^{n}\Qr{T_{k}=n}\leq (1-t_{\infty})^{K}\enspace.
\]
Therefore, from Facts \ref{fact.F2} and \ref{fact.F3},
\begin{equation}\label{control1}
\Qr{H_{n}=0}\leq \inf_{K=1,\ldots,n}\set{K^{2}(1-r_{n/K})+(1-t_{\infty})^{K}}\enspace.
\end{equation}

\end{proof}
%
%
%
%
%
\begin{proof}[Proof of Item (iii) of Proposition \ref{prop:expli}]

In this section, we assume that, for all $k$, $1-r_{k}\leq C_{r}r^{k}$, for some $r\in(0,1)$ and a constant $C_{r}>0$. In that case, for all $k$, we have, by independence,
\begin{align*}
\Qr{\sum_{l=1}^{k}I_{l}=n}&=\sum_{i_{1}+\ldots+i_{k}=n}\Qr{\bigcap_{l=1}^{k} I_{l}=i_{l}}\\
&=\sum_{i_{1}+\ldots+i_{k}=n}\prod_{l=1}^{k}\Qr{I_{l}=i_{l}}\\
&\leq \sum_{i_{1}+\ldots+i_{k}=n}C_{r}^{k}r^{i_{1}+\ldots+i_{k}}=C_{r}^{k}r^{n}\sum_{i_{1}+\ldots+i_{k}=n}1\enspace.
\end{align*}
Let us evaluate the numbers $p_{k,n}=\sum_{i_{1}+\ldots+i_{k}=n}1$. We have $p_{1,n}=1$ and
\begin{align*}
p_{k,n}&=\sum_{l=1}^{n-k+1}\sum_{i_{k}=l}\sum_{i_{1}+\ldots+i_{k-1}=n-l}1=\sum_{l=1}^{n-k+1}p_{1,l}p_{k-1,n-l}\\
&=\sum_{l=1}^{n-k+1}p_{k-1,n-l}\enspace.
\end{align*}
Let us then assume that, for some $k$, we have, for all $n\geq k-1$, $p_{k-1,n}\leq n^{k-2}/(k-2)!$. Notice that this is the case for $k=2$, then, for all $n\geq k$,
\begin{align*}
p_{k,n}&\leq \sum_{l=1}^{n-k+1}\frac{(n-l)^{k-2}}{(k-2)!}=\sum_{l=k-1}^{n-1}\frac{l^{k-2}}{(k-2)!}\leq \int_{k-1}^{n}\frac{x^{(k-2)}}{(k-2)!}\leq \frac{n^{k-1}}{(k-1)!}\enspace.
\end{align*}
We deduce that
\[
\sum_{k=1}^{n}C_{r}^{k}\sum_{i_{1}+\ldots+i_{k}=n}1\leq \frac1{C_{r}}\sum_{k=1}^{n}\frac{(C_{r}n)^{k-1}}{(k-1)!}\leq \frac{e^{C_{r}n}}{C_{r}}\enspace.
\]
Therefore, 
\[
\Qr{H_{n}=0}=\sum_{k=1}^{n}\Qr{T_{k}=n}\leq \frac1{C_{r}}(e^{C_{r}}r)^{n}\enspace.
\]
Hence, when $C_{r}<\ln (1/r)$, $e^{C_{r}}r<1$ and $\Qr{H_{n}=0}$ decreases exponentially fast.
\end{proof}

\section{Concentration of geometric random variables}\label{sec.conc.geom}
Let $\xi, \xi_{1:n}$ be i.i.d. geometric random variables with parameter $\alpha$, i.e., $\forall k\geq 1$, $\P\paren{\xi=k}=(1-\alpha)^{k-1}\alpha$. We obtain in this section the following upper bounds.
\begin{prop}
 let $C_{1,\alpha}=\frac{1-\alpha}{\alpha}+4\paren{\frac{1-\alpha}{\alpha}}^{2}$, $C_{2,\alpha}=\ln\paren{\frac{2-\alpha}{2(1-\alpha)}\wedge 2}$. Then, $\forall x>0$,
\begin{align}
\label{eq.conc.geo.up}\P\paren{\frac1n\sum_{i=1}^{n}X_{i}-\frac1{\alpha}>x}\leq  e^{-n\paren{\frac{x^{2}}{2C_{1,\alpha}}\wedge \frac{C_{2,\alpha}}2x}}\enspace.\\
\notag\P\paren{\frac1n\sum_{i=1}^{n}X_{i}-\frac1{\alpha}<-x}\leq e^{-n\paren{\frac{x^{2}}{2C_{1,\alpha}}\wedge \frac{x}2}}\enspace. 
\end{align}
\end{prop}
As a corollary of this result, we obtain the following bound when $n=\lfloor k\alpha/2\rfloor$ and $x=1/\alpha$.
\begin{coro}\label{coro.contuk}
Let $k\in \N^{*}$, $\alpha\in(0,1)$, $n=\lfloor k\alpha/2\rfloor$, $x=k/(2n)\geq 1/\alpha$, $\xi_{1:n}$ be i.i.d. random variables with parameters $\alpha$, and 
\[u_{k}:= n\P\paren{\absj{\sum_{j=1}^{n}\xi_{j}-\frac{n}{\alpha}}>nx}\enspace.\]
Then, we have, for $C_{3,\alpha}\egaldef \frac{\alpha}{4(1-\alpha)(4-3\alpha)}\wedge\frac1{4}\ln\paren{\frac{2-\alpha}{2(1-\alpha)}\wedge 2}$, for all $\epsilon>0$ and all $k>k(\epsilon)$,
\[
u_{k}\leq \alpha e^{-k(C_{3,\alpha}-\epsilon)}\enspace.
\]
\end{coro}
\subsection{Chernov's bound}
Let $Y,Y_{1:n}$ be i.i.d. random variables such that $\forall a<\lambda<b$, $\E\paren{e^{\lambda Y}}<\infty$, then, 
\begin{equation}\label{eq.chernov.bound}
\forall x>0,\qquad \P\paren{\frac1n\sum_{i=1}^{n}Y_{i}>x}\leq \inf_{na<\lambda<nb}e^{-\lambda x}\paren{\E\paren{e^{\frac{\lambda}n Y}}}^{n}\enspace.
\end{equation}
\begin{proof}
We have, by independence of the $Y_{i}$ and Markov's inequality, for all $na<\lambda<nb$,
\begin{align*}
 \P\paren{\frac1n\sum_{i=1}^{n}Y_{i}>x}&=\P\paren{e^{\frac{\lambda}n\sum_{i=1}^{n}Y_{i}}>e^{\lambda x}}\leq e^{-\lambda x}\E\paren{e^{\frac{\lambda}n\sum_{i=1}^{n}Y_{i}}}=e^{-\lambda x}\paren{\E\paren{e^{\frac{\lambda}n Y}}}^{n}\enspace.
\end{align*}
\end{proof}
\subsection{Exponential moments of geometric random variables}
Let $\xi$ be a geometric random variable with parameter $\alpha$, then 
\begin{align}\label{eq.moment.expo.geo}
\forall \lambda<-\ln(1-\alpha),&\qquad \E\paren{e^{\lambda \xi}}\leq \frac{\alpha e^{\lambda}}{1-(1-\alpha)e^{\lambda}}\enspace,\\
\notag\forall \lambda>\ln(1-\alpha),&\qquad \E\paren{e^{\lambda(-\xi)}}\leq \frac{\alpha e^{-\lambda}}{1-(1-\alpha)e^{-\lambda}}\enspace.
\end{align}
\begin{proof}
By definition, we have, $\forall\lambda<-\ln(1-\alpha)$,
\begin{align*}
\E\paren{e^{\lambda \xi}}&=\sum_{k\geq 1}e^{\lambda k}(1-\alpha)^{k-1}\alpha=\alpha e^{\lambda}\sum_{k\geq 0}\paren{(1-\alpha)e^{\lambda}}^{k}=\frac{\alpha e^{\lambda}}{1-(1-\alpha)e^{\lambda}}\enspace.
\end{align*}
Moreover, for all $\lambda>\ln(1-p)$,
\begin{equation*}
\E\paren{e^{-\lambda \xi}}=\alpha e^{-\lambda}\sum_{k\geq 0}\paren{(1-\alpha)e^{-\lambda}}^{k}=\frac{\alpha e^{-\lambda}}{1-(1-\alpha)e^{-\lambda}}\enspace.
\end{equation*}
\subsection{Proof of the deviation bounds}
Plugging \eqref{eq.moment.expo.geo} in \eqref{eq.chernov.bound}, we obtain, for all $\lambda<-n\ln (1-\alpha)$,
\begin{align*}
\P\paren{\frac1n\sum_{i=1}^{n}\xi_{i}>\frac1{\alpha}+x}&\leq e^{-\lambda\paren{\frac1{\alpha}+x}}\paren{\frac{\alpha e^{\lambda/n}}{1-(1-\alpha)e^{\lambda/n}}}^{n}\\
&=\alpha^{n}e^{-\lambda\paren{\frac1{\alpha}+x-1}}e^{-n\ln\paren{1-(1-\alpha)e^{\lambda/n}}}
\end{align*}
Choosing $\lambda=n\epsilon$ for $\epsilon\leq \ln\paren{\frac{2-\alpha}{2(1-\alpha)}\wedge 2}$, using the inequalities $e^{\epsilon}\leq 1+\epsilon+\epsilon^{2}$ for all $\epsilon\leq \ln 2$ and $-\ln (1-u)\leq 1+u+u^{2}$ when $u\leq 1/2$, this last bound is equal to 
\begin{align*}
 &\paren{\alpha e^{-\epsilon\paren{\frac1{\alpha}+x-1}}e^{-\ln\paren{1-(1-\alpha)e^{\epsilon}}}}^{n}\\
 &\leq\paren{\alpha e^{-\epsilon\paren{\frac1{\alpha}+x-1}}e^{-\ln(\alpha)-\ln\paren{1-\frac{(1-\alpha)}{\alpha}(e^{\epsilon}-1)}}}^{n}\leq\paren{e^{-\epsilon\paren{\frac1{\alpha}+x-1}}e^{\frac{(1-\alpha)}{\alpha}(e^{\epsilon}-1)+\paren{\frac{(1-\alpha)}{\alpha}(e^{\epsilon}-1)}^{2}}}^{n}\\
 &\leq e^{-n\epsilon \paren{x-\epsilon\paren{\frac{1-\alpha}{\alpha}+4\paren{\frac{1-\alpha}{\alpha}}^{2}}}}\enspace.
\end{align*}
Let $C_{\alpha}=\frac{1-\alpha}{\alpha}+4\paren{\frac{1-\alpha}{\alpha}}^{2}$, choosing $\epsilon\leq x/(2C_{\alpha})$, we have $x-\epsilon C_{\alpha}\geq x/2$, hence, choosing $\epsilon=\frac{x}{2C_{\alpha}}\wedge \ln\paren{\frac{2-\alpha}{2(1-\alpha)}\wedge 2}$, we conclude the proof.
Plugging \eqref{eq.moment.expo.geo} in \eqref{eq.chernov.bound}, we obtain, for all $\lambda>n\ln (1-\alpha)$,
\begin{align*}
\P\paren{\frac1n\sum_{i=1}^{n}\xi_{i}<\frac1{\alpha}-x}&=\P\paren{\frac1n\sum_{i=1}^{n}(-\xi_{i})>-\frac1{\alpha}+x}\\
&\leq e^{-\lambda\paren{-\frac1{\alpha}+x}}\paren{\frac{\alpha e^{-\lambda/n}}{1-(1-\alpha)e^{-\lambda/n}}}^{n}\\
&=\alpha^{n}e^{-\lambda\paren{-\frac1{\alpha}+x+1}}e^{-n\ln\paren{1-(1-\alpha)e^{-\lambda/n}}}
\end{align*}
Choosing $\lambda=n\epsilon$, with $\epsilon\leq 1$, this last bound is equal to
\begin{align*}
 &\paren{\alpha e^{-\epsilon\paren{-\frac1{\alpha}+x+1}}e^{-\ln\paren{1-(1-\alpha)e^{-\epsilon}}}}^{n}\\
 &=\paren{\alpha e^{-\epsilon\paren{-\frac1{\alpha}+x+1}}e^{-\ln\paren{\alpha}-\ln\paren{1-(e^{-\epsilon}-1)\frac{1-\alpha}{\alpha}}}}^{n}\leq \paren{e^{-\epsilon \paren{-\frac1{\alpha}+x+1}}e^{(e^{-\epsilon}-1)\frac{1-\alpha}{\alpha}+\paren{(e^{-\epsilon}-1)\frac{1-\alpha}{\alpha}}^{2}}}^{n}\\
 &\leq \paren{e^{-\epsilon \paren{-\frac1{\alpha}+x+1}}e^{(-\epsilon+\epsilon^{2})\frac{1-\alpha}{\alpha}+\paren{(-\epsilon+\epsilon^{2})\frac{1-\alpha}{\alpha}}^{2}}}^{n}\leq e^{-n\croch{\epsilon x-\epsilon^{2}\paren{\frac{1-\alpha}{\alpha}+4\paren{\frac{1-\alpha}{\alpha}}^{2}}}}\enspace.
\end{align*}
Let $C_{\alpha}=\frac{1-\alpha}{\alpha}+4\paren{\frac{1-\alpha}{\alpha}}^{2}$, choosing $\epsilon\leq x/(2C_{\alpha})$, we have $x-\epsilon C_{\alpha}\geq x/2$, hence, choosing $\epsilon=\frac{x}{2C_{\alpha}}\wedge 1$, we conclude the proof.
\end{proof}


\end{document}